\documentclass[12pt]{article}
\usepackage[T2A]{fontenc}
\usepackage[english]{babel}
\usepackage{amssymb,amsmath,amsfonts,amsthm,indentfirst,bm}
\usepackage{cite}
\usepackage[a4paper,mag=1000,includefoot,left=2cm,right=2cm,top=2cm,bottom=2cm,headsep=1cm,footskip=1cm]{geometry}
          \def\b{\beta}
\def\g{\gamma}          \def\G{\Gamma}
\def\de{\delta}         \def\De{\Delta}
\def\ve{\varepsilon}    \def\th{\theta}
\def\vk{\varkappa}
\def\la{\lambda}        

\def\s{\sigma}          

            \def\vp{\varphi}
\def\Om{\Omega}
\def\bu{{\bf u}}        
\def\hA{\hat{A}}        
\def\hf{\hat{f}}        \def\hg{\hat{g}}
\def\hp{\hat{p}}        
\def\hu{\hat{u}}
\def\hV{\hat{V}}
\def\hw{\hat{w}}
\def\hx{\hat{x}}
\def\hrho{\hat{\rho}}   
\def\het{\hat{\eta}}    \def\hth{\hat{\theta}}
\def\hpi{\hat{\pi}}     \def\hsi{\hat{\sigma}}
\def\be{\begin{equation}}   \def\ee{\end{equation}}
\def\lan{\langle}       \def\ran{\rangle}
\def\leq{\leqslant}     \def\geq{\geqslant}
\def\half{\tfrac12}

\newtheorem{theorem}{\indent Theorem}
\newtheorem{proposition}{\indent Proposition}
\newtheorem{lemma}{\indent Lemma}
\newtheorem{corollary}{\indent Corollary}
\newtheorem{rem}{\indent Remark}
\begin{document}
\centerline{\bf ON WEAK SOLUTIONS TO THE 1D COMPRESSIBLE NAVIER-STOKES}
\smallskip
\centerline{\bf EQUATIONS: A LIPSCHITZ CONTINUOUS DEPENDENCE ON DATA }
\smallskip
\centerline{\bf IN WEAKER NORMS AND AN ERROR OF THEIR HOMOGENIZATION}

\begin{center}{Alexander Zlotnik}
\end{center}
\centerline{Higher School of Economics University,
Pokrovskii bd. 11, 109028 Moscow, Russia}
\smallskip
\centerline{E--mail address: azlotnik@hse.ru}

\renewcommand{\abstractname}{}
\begin{abstract}
\noindent \textbf{Abstract}. We deal with the global in time weak solutions to the 1D compressible Navier-Stokes system of equations for large discontinuous initial data and nonhomogeneous boundary conditions of three standard types. 
We prove the Lipschitz-type continuous dependence of the solution $(\eta,u,\theta)$,
in a norm slightly stronger than $L^{2,\infty}(Q)\times L^2(Q)\times L^2(Q)$,  on the initial data $(\eta^0,u^0,e^0)$ in a norm of
$L^2(\Omega)\times H^{-1}(\Omega)\times H^{-1}(\Omega)$--type  and also on the free terms in all the equations in some dual norms.
Here $\eta$, $u$ and $\theta$ are the specific volume, velocity and absolute temperature as well as $\eta^0$, $u^0$ and $e^0$ are the initial specific volume, velocity and specific total energy, and $Q=\Omega\times (0,T)$.
We also apply this result to the case of discontinuous rapidly oscillating, with the period $\ve$, initial data and free terms and derive an estimate $O(\ve)$ for the difference between the solutions to the Navier-Stokes equations and their Bakhvalov-Eglit two-scale homogenized version with averaged data.


\end{abstract}
\bigskip

\section{\normalsize\rm\hspace{3pt}\bf Introduction}

The famous well-posedness theory of global in time strong solutions to initial-boundary value problems (IBVPs) for the 1D compressible Navier-Stokes system of equations for large initial data from the Sobolev space $H^1(\Omega)$ was developed by A.V. Kazhikhov near 1980 year and presented in detail in \cite[Ch. 2]{AKM83}.
For the large discontinuous initial data from the Lebesgue spaces, the corresponding well-posedness theory of global in time weak solutions was constructed in \cite{AZSMD89,AZMN92,AZRM97,ZASMJ97}, see also the related papers \cite{Ser86,H92} and the more recent review \cite{ZSprin17}.
Notice that deep investigations on global weak solutions in multidimensional case were carried out, in particular, see \cite{F04,FN09}, but their uniqueness has not yet been fully explored.

\par The latter theory contains several theorems on a Lipschitz-type continuous dependence on data, see \cite{AZSMD89,ZASMJ97,ZAMN98}.
The first of them \cite[Theorem 3(a)]{AZSMD89} was concerned the Lipschitz-type continuous dependence of the solution $(\eta,u,\theta)$,
in a norm slightly stronger than $L^{2,\infty}(Q)\times L^2(Q)\times L^2(Q)$,  on the initial data $(\eta^0,u^0,e^0)$ in a norm of
$L^2(\Omega)\times H^{-1}(\Omega)\times H^{-1}(\Omega)$--type  and also on the free terms in the momentum and internal energy equations in some dual norms (as well as on the boundary data in $W^{1,1}(0,T)$ or $L^1(0,T)$--norms).
Here $\eta$, $u$ and $\theta$ are the specific volume, velocity and absolute temperature as well as $\eta^0$, $u^0$ and $e^0$ are the initial specific volume, velocity and specific
total energy, and $Q=\Omega\times (0,T)$.
The complicated proof of this subtle theorem has not been  published before.
Note that the proof imposes stronger Lebesgue conditions on the initial data than in \cite{AZRM97,ZASMJ97}: $u^0\in L^\infty(\Omega)$ and $\theta^0\in L^2(\Omega)$, where $\theta^0$ is the initial absolute temperature.

\par The other proved theorems on the Lipschitz continuous dependence dealt with  much stronger norms of the solution like $L^\infty(Q)\times V_2(Q)\times V_q(Q)$ but on the initial data $(\eta^0,u^0,\theta^0)$ also in much stronger norms like
$L^\infty(\Omega)\times L^2(\Omega)\times L^q(\Omega)$, $q=1,2$, and the respective stronger norms of the free terms and boundary data.
Some other theorems on the same subject for weak solutions to the Cauchy problem were proved in \cite{H96,JZ04}.

\par In this paper, we present a more general theorem than \cite[Theorem 3(a)]{AZSMD89}, with additional perturbations in the mass, momentum and internal energy balance equations and the more general free terms in the momentum and energy balance equations depending also on the Eulerian coordinate (that is natural physically), and give  its full proof.
Notice that we treat simultaneously several IBVPs with three standard types of boundary conditions for $u$ mainly in the unified manner (the precise statements depend on the type of boundary conditions).
In addition, we derive corollaries on the H\"{o}lder continuity on the data in stronger norms, in general, under some additional assumptions on the data.

\par Homogenization theory is well-developed for various time-dependent PDEs, in particular, see \cite{AAD21,BB23,CD99,GVK20,T10,ZhKO94}.
In the second part of the paper, we consider the 1D compressible Navier-Stokes equations with discontinuous rapidly oscillating, with the period $\ve$, initial data and free terms.
We apply the new Lipschitz continuity theorem to derive an estimate $O(\ve)$ for the difference between the solutions to these equations and their Bakhvalov-Eglit two-scale homogenized version with averaged data; note a non-trivial averaging of $\th^0$.
The presence of the above mentioned additional perturbations is essential here.
In addition, we also obtain the smaller order estimates $O(\ve^{1/2})$ and $O(\ve^{1/4})$ for the error components in the stronger norms $C(0,T;L^2(\Om))$ and $L^\infty(Q)$ (or $C(\bar{Q})$). Notice that the homogenization error estimation is only one of possible applications of the proved theorem on Lipschitz continuity dependence on data.

\par Concerning the formal derivation of the Bakhvalov-Eglit equations, see \cite{BEDAN83} and \cite{E92}.
Their rigorous justification (in the more complicated case than here of nonhomogeneous gas with the rapidly oscillating gas parameters) was given much later in \cite{AZCMMP98}.
Previously the simpler 1D barotropic case had been studied in \cite{AZCMMP96a,AZCMMP96b,S91}, in particular, the corresponding  estimates $O(\ve)$ were derived in \cite{AZCMMP96b}.
We essentially use the techniques from these papers (the general techniques goes back to \cite{BZSMD78}).
See also \cite{AZCRASP01,S14} and \cite{AG07} on the related mathematical studies of  thermo-viscoelasticity and visco-elastoplasticity models as well as \cite{E14} on the form of equations for models of such kind.
We also mention recent studies \cite{B23,B24,B25} and \cite{H23a,H23b} concerning compressible binary mixtures and bubbly flows involving related homogenization ideas.

\par The paper is organized as follows. 
In Section 2, the main notation and important  auxiliary bounds of weak solutions to the 1D linear parabolic IBVPs are given. 
Section 3 is devoted to the statements of the IBVPs for the 1D compressible Navier-Stokes equations and its perturbation as well as a five-step lengthy proof of the Lipschitz-type continuity theorem of weak solutions on the data (the first main result of the paper).
Two its corollaries on the H\"{o}lder-type bounds in stronger norms (including those interior in time) are also presented.
In Section 4, first some auxiliary results on periodic rapidly oscillating functions are given. 
Then the Navier-Stokes equations with the discontinuous rapidly oscillating data and the corresponding Bakhvalov-Eglit two-scale homogenized equations with averaged data are considered, and the theorem on $O(\ve)$--bound for the difference between the solutions to these equations is proved (the second main result of the paper), and its corollary on smaller order bounds in stronger norms (including the uniform one) is given.

\section{\normalsize\rm\hspace{3pt}\bf The notation and auxiliary linear parabolic results}
\label{sect2}
We use the standard Lebesgue and Sobolev spaces $L^q(D)$, $W^{1,q}(D)$, with $q\in [1,\infty]$, and $H^1(D)=W^{1,2}(D)$ for
$D=\Omega:=(0,X),(0,T)$ and $Q=Q_T:=\Omega\times (0,T)$ as well as the anisotropic Lebesgue space $L^{q,r}(Q_T)$, with $q,r\in [1,\infty]$, equipped with the norm $\|w\|_{L^{q,r}(Q_T)}=\|\|w(x,t)\|_{L^q(\Om)}\|_{L^r(0,T)}$.
For a Banach space $B$, let $L^\infty(0,T;B)$ be the space of strongly measurable functions $w$: $(0,T)\to B$ having the finite norm $\|w\|_{L^\infty(0,T;H)}=\|\|w(t)\|_B\|_{L^\infty(0,T)}$, and  $C(0,T;B)$ be the space of continuous functions  $w$: $[0,T]\to B$ equipped with the norm $\|w\|_{C(0,T;B)}=\max_{0\leq t\leq T}\|w(t)\|_B$.
In the last section, similar spaces for $(0,T)$ replaced with $\Om$ will be also used.
A norm of a vector-function is understood as the sum of the same norms of its components.
 
\par For $y\in L^1(\Om)$, we define the primitive functions and the mean value over $\Om$:
\[
Iy(x):=\int_0^xy(\zeta)\,d\zeta,\ \
I^*y(x):=\int_x^Xy(\zeta)\,d\zeta,\ \
 \lan y\ran_\Om:=\frac{1}{X}\int_\Om y(x)\,dx=\frac1X Iy(X).
\]
We also introduce the integral operators
\[
 I^{\lan 1\ran}y:=Iy-\lan Iy\ran_\Om,\ \ I^{\lan 2\ran}:=I,\ \
 I^{\lan 3\ran}y:=I(y-\lan y\ran_\Om).
\]
Clearly $(I^{\lan 2\ran}y)|_{x=0}=0$, $(I^{\lan 3\ran}y)|_{x=0,X}=0$ as well as  
\be
 I^{\lan 1\ran}Ds=s-\lan s\ran_\Om,\ \
 I^{\lan 2\ran}Ds=s-s(0),\ \
 I^{\lan 3\ran}Ds(x)=s(x)-(1-\tfrac{x}{X})s(0)-\tfrac{x}{X}s(X)
\label{idp}
\ee
for $s\in W^{1,1}(\Om)$.
For $y,z\in L^1(\Om)$, the following integral identities hold
\begin{gather}
 \int_\Om (Iy)z\,dx=\int_\Om yI^*z\,dx,
\label{int id with I 1}\\
 \int_\Om (I^{\lan 1\ran}y)z\,dx
 =\int_\Om (Iy)(z-\lan z\ran_\Om)\,dx
 =\int_\Om yI^*(z-\lan z\ran_\Om)\,dx
 =-\int_\Om yI^{\lan 3\ran}z\,dx,
\label{int id with I 2}
\end{gather}
and thus $(I^{\lan 2\ran})^*=I^*$ and $(I^{\lan 3\ran})^*=-I^{\lan 1\ran}$.

\par Let $H^{1;m}$ be the subspace of functions $\vp\in H^1(\Om)$ with $\vp|_{x=0,X}=0$ for $m=1$ or $\vp|_{x=X}=0$ for $m=2$ as well as
$H^{1;3}=H^1(\Om)$. For $y\in L^1(\Om)$, we introduce the norms equivalent to the dual norms in $H^{1;m}$
\[
 \|y\|_{H^{-1;m}}:=\|I^{\lan m\ran}y\|_{L^2(\Om)}\ \ (m=1,2),\ \ 
\|y\|_{H^{-1;m}}:=\|Iy\|_{L^2(\Om)}+X|\lan y\ran_\Om|\ \ (m=3).
\]

\par  We define also the primitive in $t$ function
$I_tb(t):=\int_0^tb(\tau)\,d\tau$, for $b\in L^1(0,T)$.

\par Recall the well-known space $V_2(Q)$ of functions $\vp\in L^{2,\infty}(Q)$ with $D\vp\in L^2(Q)$ equipped with the norm
$\|\vp\|_{V_2(Q)}=\|\vp\|_{L^{2,\infty}(Q)}+\|D\vp\|_{L^2(Q)}$.

\par Let in the section $\vk>0$ in $Q$ and $\vk,\frac{1}{\vk}\in L^\infty(Q)$; below it is essential that there exists no derivative $D\vk$ that complicates our definitions and results.
Let $V_{2;\vk,m_*}(Q)$ be the subspace of functions $\vp\in V_2(Q)$ such that
\[
 \lan\tfrac{\vp}{\vk}(\cdot,t)\ran_\Om=0\ \ \text{on}\ \ (0,T)\ \ (m=1),\ \
 \vp|_{x=0}=0\ \ (m=2),\ \ \vp|_{x=0,X}=0\ \ (m=3).
\]
For $w\in L^1(Q)$, we consider the norm dual to the norm in $V_{2;\vk,m_*}(Q)$:
\[
 \|w\|_{[V_{2;\vk,m_*}(Q)]^*}
 :=\sup_{\vp\in V_{2;\vk,m_*}(Q),\,\|\vp\|_{V_2(Q)}=1}\int_Qw\vp\,dxdt\leq +\infty.
\]
For $m=2,3$, there is no dependence on $\vk$, and one can omit it in the notation.
Owing to the imbedding $V_2(Q)\subset L^{q',r'}(Q)$ \cite[Section 2.3]{LSU68} and the H\"{o}lder inequality the following inequality holds
\be
 \|w\|_{[V_{2;\vk,m_*}(Q)]^*}
 \leq \|w\|_{[V_2(Q)]^*}
 :=\sup_{\|\vp\|_{V_2(Q)}=1}\int_Qw\vp\,dxdt
 \leq C\inf_{M_0}\|w\|_{L^{q,r}(Q)},
\label{ineq3}
\ee
where $M_0$ is the set of the integrability exponents $q,r\in [1,\infty]$ such that $(2q)^{-1}+r^{-1}\leq 5/4$, and $\inf_{M_0}=\inf_{(q,r)\in M_0}$. Hereafter, for example,
$1/q'+1/q=1$.
Below we mainly need the values $(q,r)=(2,1)$, $(1,4/3)$ and $(6/5,6/5)$.
Note that this dual norm is finite for $w=w_1+\ldots+w_k$  such that $w_i\in L^{q_i,r_i}(Q)$, with $(q_i,r_i)\in M_0$, $1\leq i\leq k$, although not for any $w\in L^1(Q)$.

\par For $y,z\in L^1(\Om)$, we also need the $\frac{1}{\vk}$-weighted mean value over $\Om$ and the auxiliary operator
\be
 \lan z\ran_{\Om,1/\vk}:=\frac{\lan z/\vk\ran_\Om}{\lan 1/\vk\ran_\Om},\ \
 P_{1/\vk}y:=y-\frac{1/\vk}{\lan 1/\vk\ran_\Om}\,\lan y\ran_\Om=\frac{1}{\vk}\big(\vk y-\lan \vk y\ran_{\Om,1/\vk}\big).
\label{pvk}
\ee
Notice that $\lan P_{1/\vk}(y)\ran_\Om=0$ and the following identity holds
\be
 \int_\Om (P_{1/\vk}y)z\,dx
 =\int_\Omega\frac{1}{\vk}\big(\vk y-\lan \vk y\ran_{\Om,1/\vk}\big)z\,dx
 =\int_\Om y\big(z-\lan z\ran_{\Om,1/\vk}\big)\,dx.
\label{idenP}
\ee

\par For $w\in L^1(Q)$ such that  $\|w\|_{[V_{2;\vk,m_*}(Q)]^*}<\infty$ and $IP_{1/\vk}w\in L^2(Q)$ ($m=1$) or $(I^{\lan m\ran})^*w\in L^2(Q)$ ($m=2,3$) as well as $\vp\in V_{2;\vk,m_*}(Q)$, the integral identities hold
\begin{gather*}
 \int_Qw\vp\,dx dt=-\int_Q(I P_{1/\vk}w)D\vp\,dx dt\ (m=1),
 \int_Qw\vp\,dx dt=\int_Q\big[(I^{\lan m\ran})^*w\big]D\vp\,dx dt\ (m=2,3);
\end{gather*}
they follow from the integral identities \eqref{int id with I 1}, \eqref{int id with I 2} and \eqref{idenP} and formulas \eqref{idp}.
The last identities obviously allow us to supplement inequality \eqref{ineq3} with the following ones 
\be
 \|w\|_{[V_{2;\vk,m_*}(Q)]^*}\leq\|IP_{1/\vk}w\|_{L^2(Q)}\,\ (m=1),\ \
 \|w\|_{[V_{2;\vk,m_*}(Q)]^*}\leq\|(I^{\lan m\ran})^*w\|_{L^2(Q)}\,\ (m=2,3).
\label{ineq3a}
\ee
\par The next two propositions play an essential role below.
Let us consider 1D linear parabolic IBVP $\mathcal{L}_m$, $m=1,2,3$, such that
\begin{gather}
D_tv=Ds+F,\ \ s=\vk Dv+\psi\ \ \text{in}\ \ Q,\ \
\label{eq1l}\\[1mm]
v|_{t=0}=v^0,
\label{eq2l}\\[1mm]
v|_{x=0,X}=\mathbf{v}_b\ (m=1),\ \
s|_{x=0}=s_0,\,\ v|_{x=X}=v_X\ \ (m=2),\ \
s|_{x=0,X}=\mathbf{s}_b\ \ (m=3).
\label{bcl}
\end{gather}
Here, for example, $v|_{x=0,X}=\mathbf{v}_b$ means that $v|_{x=0}=v_0$ and
$v|_{x=X}=v_X$, as well as $\mathbf{v}_b=(v_0,v_X)$ and $\mathbf{s}_b=(s_0,s_X)$ are the pairs of boundary data;
for convenience, the functions $v_0,v_X,s_0$ and $s_X$ that are missing in the $m$th boundary condition of \eqref{bcl} are set equal to $0$.
Recall that, by definition,  \textit{the weak solution $v$ from $V_2(Q)$} \cite[Section 3.1]{LSU68} to this problem satisfies the integral identity
\begin{gather}
\int_Q\big[-vD_t\vp+(\vk Dv+\psi)D\vp-F\vp\big]\,dxdt
 =\int_\Om v^0\vp|_{t=0}\,dx
 +\int_0^T(s_X\vp|_{x=X}-s_0\vp|_{x=0})\,dt,
\label{weaksolv2}
\end{gather}
for any $\vp\in H^1(Q)$, $\vp|_{t=T}=0$, with $\vp|_{x=0,X}=0$ ($m=1$) or $\vp|_{x=0}=0$ ($m=2$), and also the boundary conditions $v|_{x=0,X}=\mathbf{v}_b$
($m=1$) or $v|_{x=X}=v_X$ ($m=2$) are valid in the sense of traces of $v\in V_2(Q)$.

\par Let $N>1$ be an arbitrarily large parameter and $K(N,T)$, possibly, with indices, be non-decreasing functions in $T$ (in the proofs, the arguments are mainly omitted).
Let $\de_{mn}$ be the Kronecker delta, i.e., $\de_{mn}=0$ for $m\neq n$ and $\de_{mm}=1$.
\begin{proposition}
\label{th1}
Let the conditions $N^{-1}\leq\vk$ and
$\|\vk\|_{L^\infty(Q)}+\|D_t\vk\|_{L^2(Q)}\leq N$ together with 
\[
\psi,F\in L^1(Q),\ \ \|\tfrac{1}{\vk}\psi\|_{[V_{2;\vk,m_*}(Q)]^*}<\infty,\ \ 
v^0\in L^1(\Om),\ \ \mathbf{v}_b\in L^{4/3}(0,T),\ \ \mathbf{s}_b\in L^1(0,T)
\]
be valid.
For a weak solution $v$ from $V_2(Q)$ to problem $\mathcal{L}_m$, $m=1,2,3$, the following bound in a weaker norm holds
\begin{gather}
\|v\|_{L^2(Q)\cap L^\infty(0,T;H^{-1;m})}
 \leq K(N,T)\Big(\|v^0\|_{H^{-1;m}}
 +\|\mathbf{v}_b\|_{L^{4/3}(0,T)}
 +\|\mathbf{s}_b\|_{L^1(0,T)}
\nonumber\\[1mm]
 +\|\tfrac{1}{\vk}\psi\|_{[V_{2;\vk,m_*}(Q)]^*}
 +\|\tfrac{1}{\vk}I^{\lan m\ran}F\|_{[V_{2;\vk,m_*}(Q)]^*}
+\de_{m3}\|\lan F\ran_\Omega\|_{L^1(0,T)}
\Big).
\label{in5}
\end{gather}
Hereafter, for example, $\|\mathbf{v}_b\|_B=\|v_0\|_B+\|v_X\|_B$ for a Banach space $B$.
\end{proposition}

\par We emphasize that hereafter imposing conditions on the data norms like above $\|\vk\|_{L^\infty(Q)}+\|D_t\vk\|_{L^2(Q)}\leq N$, we \textit{automatically} assume that $\vk\in L^\infty(Q)$ and there exists $D_t\vk\in L^2(Q)$.
This allows us to shorten assumptions in the statements in a natural way.

\par This result follows from \cite[Theorem 2.2]{AZMN94}. 
We comment that, for $m=1$, our norm of $\psi$ looks different but it is the same as in \cite{AZMN94} due to identity \eqref{idenP} for $y=\frac{\psi}{\vk}$.
The norm of $F$ is not the same as in \cite{AZMN94} but, owing to the presence of term $\psi$, the case of general $F$ is reduced simply to $F=0$ for $m=1,2$ and $F=F(t)$ for $m=3$ with the help of formulas
\be
 F=DI^{\lan m\ran}F\ \ (m=1,2),\ \
 F=DI^{\lan 3\ran}F+\lan F\ran_\Omega\ \ (m=3).
\label{fpsi}
\ee
Note that the presence of the space $L^\infty(0,T;H^{-1;m})$ on the left in bound \eqref{in5} is essential below. 

\par We also need to deal with \textit{the generalized IBVP $\mathcal{L}_m$}, $m=1,2,3$, involving the more general equation and initial condition than
(\ref{eq1l}) and (\ref{eq2l}), namely,
\[
 D_tv=Ds+F+D_tG,\ \ s=\vk Dv+\psi\ \ \text{in}\ \ Q,\ \
 (v-G)|_{t=0}=v^0
\]
and the same boundary conditions \eqref{bcl}.
Following \cite{LSU68}, a function $v\in L^2(Q)$ is called its {\it weak solution from} $L^2(Q)$ if it satisfies the integral identity
\begin{gather}
 -\int_Q v\big[D_ty+D(\vk Dy)\big]\,dx dt
 =\int_Q\big(-\psi Dy+Fy-GD_ty\big)\,dx dt+\int_\Om v^0y|_{t=0}\,dx
\nonumber\\
+\int_0^T\big\{s_X y|_{x=X}-s_0 y|_{x=0}
 -\big[v_X(\vk Dy)|_{x=X}-v_0(\vk Dy)|_{x=0}\big]\big\}\,dt
\label{idenL2}
\end{gather}
for any $y\in H^{2,1;\vk,m}(Q)$; the definition is correct for $G\in L^2(Q)$.

\par Here $H^{2,1;\vk,m}(Q)$ is the Banach space of functions $y\in H^1(Q)$ such that $\vk Dy\in V_2(Q)$ and
\begin{gather*}
y|_{x=0,X}=0\ (m=1),\ \
 (\vk Dy)|_{x=0}=0,\ \ y|_{x=X}=0\ (m=2),\ \
 (\vk Dy)|_{x=0,X}=0\ (m=3)
\end{gather*}
as well as $y|_{t=T}=0$.
Thus $\vk Dy\in V_{2;\vk,m_*}(Q)$ for $m=1,2,3$ that we use below in \eqref{in6_1}.
We define the norm in $H^{2,1;\vk,m}(Q)$ and, for $F\in L^1(Q)$, the dual one
\[
 \|y\|_{H^{2,1;\vk,m}(Q)}:=\|D_ty\|_{L^2(Q)}+\|\vk Dy\|_{V_2(Q)},\ \
 \|F\|_{[H^{2,1;\vk,m}(Q)]^*}
 :=\sup_{\|y\|_{H^{2,1;\vk,m}(Q)}=1}\int_Q Fy\,dx dt.
\]
Let $N^{-1}\leq\vk$ in $Q$.
The simple imbedding $\|y\|_{L^\infty(Q)}\leq CN\|y\|_{H^{2,1;\vk,m}(Q)}$ implies the inequality
\be
 \|F\|_{[H^{2,1;\vk,m}(Q)]^*}\leq CN\|F\|_{L^1(Q)}.
\label{in5a}
\ee
Also $y|_{x=\alpha}\in L^\infty(0,T)$ for any $0\leq\alpha\leq X$,
thus identity \eqref{idenL2} is correct for $\mathbf{s}_b\in L^1(0,T)$.
\par Analyzing the above integral, we use formulas (\ref{fpsi}), integrate by parts in $x$ and $t$ and get
\be
 \int_Q Fy\,dx dt
 =-\int_Q (I^{\lan m\ran}F)Dy\,dx dt
 -\de_{m3}\int_Q \big(I_t\lan F\ran_\Omega\big)D_ty\,dx dt,
\label{in5b}
\ee
and then supplement inequality \eqref{in5a} with the following one  
\begin{gather}
 \|F\|_{[H^{2,1;\vk,m}(Q)]^*}
 \leq \|\tfrac{1}{\vk}I^{\lan m\ran}F\|_{[V_{2;\vk,m_*}(Q)]^*}
 +\de_{m3}X^{1/2}\|I_t\lan F\ran_\Omega\|_{L^2(0,T)}.
\label{est for F}
\end{gather}
Notice that the norm $\|F\|_{[H^{2,1;\vk,m}(Q_t)]^*}$ is non-decreasing in $0<t\leq T$ (since the extension of $y\in H^{2,1;\vk,m}(Q_t)$ by zero on $Q_T\setminus Q_t$ belongs to $H^{2,1;\vk,m}(Q_T)$).
\begin{proposition}
\label{th2}
Let the conditions on the data from Proposition \ref{th1} be valid and $G\in L^2(Q)$.
Then the weak solution $v$ from $L^2(Q)$ to problem $\mathcal{L}_m$, $m=1,2,3$, exists, is unique and obeys the bound
\begin{gather}
\|v\|_{L^2(Q)}
 \leq K(N,T)\big(\|v^0\|_{H^{-1;m}}
 +\|\mathbf{v}_b\|_{L^{4/3}(0,T)}+\|\mathbf{s}_b\|_{L^1(0,T)}
\nonumber\\[1mm]
 +\|\tfrac{1}{\vk}\psi\|_{[V_{2;\vk,m_*}(Q)]^*}
 +\|F\|_{[H^{2,1;\vk,m}(Q)]^*}+\|G\|_{L^2(Q)}\big).
\label{in6}
\end{gather}
\end{proposition}

\par The proposition is a consequence of a more general result \cite[Theorem 2.6]{ZASMJ97}. 
We comment that the term with $\psi$ appears in \eqref{in6} due to the inequality
\begin{gather}
\sup_{\|y\|_{H^{2,1;\vk,m}(Q)}=1}\int_Q \psi Dy\,dx dt\leq\sup_{\|z\|_{V_{2;\vk,m_*}(Q)}=1}\int_Q \tfrac{1}{\vk}\psi z\,dx dt=\|\tfrac{1}{\vk}\psi\|_{[V_{2;\vk,m_*}(Q)]^*}.
\label{in6_1}
\end{gather}
Also, according to inequality \eqref{est for F}, the norm of $F$ in bound \eqref{in6} is weaker than in  \eqref{in5}.

\par Notice that we could consider some data as distributions instead of $L^1$--functions in both Propositions \ref{th1} and \ref{th2} as well as  state Proposition \ref{th1} as the existence and uniqueness result for a weaker solution than from $V_2(Q)$.
But we do not need such generalizations below, and thus we avoid them to clarify and simplify the presentation and proof of the main result.

\par Below to bound solutions of integral inequalities in time, we often apply the well-known Gron\-wall--Bellman lemma, for example, see, \cite[Section II.2]{AKM83} and the next one which is a particular case of \cite[Lemma 2.2]{AZMN94}.
\begin{lemma}
\label{lem1}
Let $1\leq r_1<r\leq\infty$.
If $y\in L^r(0,T)$ satisfies the integral inequality
\[
\|y\|_{L^r(0,t)}\leq N\|y\|_{L^{r_1}(0,t)}+C_0\ \ \text{for any}\ \ 0< t\leq T,
\]
with some constant $C_0\geq 0$, then $\|y\|_{L^r(0,T)}\leq K(N,T)C_0$.
Here $K(N,T)$ depends only on $N$, $T$, $r$ and $r_1$ and is non-decreasing in $T$.
\end{lemma}
\section{\normalsize\rm\hspace{3pt}\bf The initial-boundary value problems for the 1D compressible Navier-Stokes system of equations}
\label{sect3}
The 1D compressible Navier-Stokes system of equations describes 1D flows of a viscous heat-conducting gas and
consists of the mass, momentum and internal energy balance equations
\begin{gather}
D_t\eta=Du,
\label{eq1}\\
D_tu=D\s+g[x_e],
\label{eq2}\\
c_VD_t\theta=D\pi+\s Du+f[x_e],
\label{eq3}
\end{gather}
together with the equation and formulas
\begin{gather}
 D_tx_e=u,\ \ \s=\nu\rho Du-p,\ \ \rho=\tfrac{1}{\eta},\ \
 p=k\rho\theta,\ \ \pi=\la\rho D\theta
\label{eq5}
\end{gather}
in the Lagrangian mass coordinates $(x,t)\in\bar{Q}$.
The sought functions $\eta>0$, $u$, $\theta>0$ and $x_e$ are the specific volume, velocity, absolute temperature and the Eulerian coordinate.
The functions $\s$, $\rho$, $p$ and $-\pi$ are the stress, density, pressure and heat flux.
Moreover, $g(\chi,x,t)$ is the density of body forces, and $f(\chi,x,t)\geq 0$ is the intensity of heat sources as well as $h[x_e](x,t)=h(x_e(x,t),x,t)$ for $h=g,f$.
Also $\nu>0$, $k>0$ and $\la>0$ are physical constants.

\par The system is supplemented with the initial conditions and non-homogeneous boundary conditions
\begin{gather}
(\eta,u,\theta,x_e)|_{t=0}=(\eta^0,u^0,\theta^0,x_e^0),
\label{ic}\\[1mm]
u|_{x=0,X}=\mathbf{u}_b\ (m=1),\ \
\s|_{x=0}=-p_0,\ u|_{x=X}=u_X\ (m=2),\ \
\s|_{x=0,X}=-\mathbf{p}_b\ (m=3),
\label{bc}\\[1mm]
\pi|_{x=0,X}={\bm{\pi}}_b.
\label{bc4}
\end{gather}
Here $x_e^0=I\eta^0$, and $\mathbf{u}_b=(u_0,u_X)$, $\mathbf{p}_b=(p_0,p_X)$ and $\bm{\pi}_b=(\pi_0,\pi_X)$ are the pairs of boundary data;
the functions $u_0,u_X,p_0$ and $p_X$ that are missing in the $m$th boundary condition of \eqref{bc} are set equal to $0$.
Recall that
$u_0$ and $u_X$, $p_0$ and $p_X$ as well as $-\pi_0$ and $-\pi_X$ are respectively the given velocities, outer pressures as well as heat fluxes at the left and right boundaries all depending on $t$.
We denote IBVP (\ref{eq1})--(\ref{bc4}) by $\mathcal{P}_m$.

\par Below we need the following conditions on the initial data, the terms $g$ and $f$ and the boundary data, respectively. 
Recall that $N>1$ is a parameter.
\par $(C_1)$ Let $N^{-1}\leq\eta^0$ and $0<\th^0$ on $\Om$ and
\[
\|\eta^0\|_{L^\infty(\Om)}+\|u^0\|_{L^\infty(\Om)}+\|\th^0\|_{L^2(\Om)}+\|\ln\th^0\|_{L^1(\Om)}\leq N.
\]
The term with $\ln\th^0$ can be omitted provided that $N^{-1}\leq\th^0$ on $\Om$.

\par $(C_2)$ Let $g(\chi,x,t)$ and $f(\chi,x,t)$ be  measurable on $\mathbb{R}\times Q$ and satisfy there the bounds
$|g(\chi,x,t)|\leq\bar{g}(x,t)$ and $0\leq f(\chi,x,t)\leq\bar{f}(x,t)$ with
$\|\bar{g}\|_{L^2(Q)}\leq N$ and $\|\bar{f}\|_{L^{2,1}(Q)}\leq N$.
Let also, for some $2\leq q_e\leq\infty$, the Lipschitz-type condition containing the Sobolev derivatives in $\chi$ hold
\be
 \|D_\chi g\|_{L^{\infty,q_e',1}((-a,a)\times Q)}\leq C_0(a),\ \
 \|D_\chi f\|_{L^{\infty,q_e',1}((-a,a)\times Q)}\leq C_0(a)\ \ \text{for any}\ \ a>1.
\label{ggff}
\ee
\par $(C_3)$ Let $\|\mathbf{u}_b\|_{W^{1,1}(0,T)}+\|\mathbf{p}_b\|_{W^{1,1}(0,T)}+\|\bm{\pi}_b\|_{L^{4/3}(0,T)}\leq N$, where $p_{\alpha}\geq 0$ and $\pi_{\alpha}\geq 0$ on $(0,T)$, for $\alpha=0,X$.
Let also 
\begin{gather}
N^{-1}\leq\|\eta^0\|_{L^1(\Om)}+I_t(u_X-u_0)\ \ \text{on}\ \ (0,T)\ \ (m=1),
\label{gas volume}\\
p_\alpha(t+\tau)-p_\alpha(t)\leq a_\tau(t)\int_t^{t+\tau}p_\alpha(q)\,dq
\ \ \text{for}\ \ 0<t<T-\tau,\ \ \alpha=0,X\ \  (m=2,3),
\label{cond on p}
\end{gather}
with some $a_\tau\geq 0$ on $(0,T-\tau)$, for any $0<\tau<T$, and $\sup_{0<\tau<T}\|a_\tau\|_{L^1(0,T-\tau)}\leq K_0(N)$.

\par Condition \eqref{gas volume} is necessary since the following well-known formula for the gas volume holds (arising after integration of equation \eqref{eq1} over $Q_t$)
\be
V(t):=\|\eta(\cdot,t)\|_{L^1(\Om)}=\|\eta^0\|_{L^1(\Om)}+I_t(u_X-u_0)\ \ \text{on}\ \ [0,T]\ \ (m=1).
\label{meta}
\ee
Also it is not difficult to check that condition \eqref{cond on p} is valid, in particular, if $N^{-1}\leq p_\alpha$ on $(0,T)$, or $D_tp_\alpha\leq K_0(N)p_\alpha$ on $(0,T)$.

\par Recall that, under conditions $(C_1)$--$(C_3)$, there exists a weak solution to problem $\mathcal{P}_m$ such that
$\eta,1/\eta\in L^\infty(Q)$, $D_t\eta\in L^2(Q)$, $u,u^2,\th\in V_2(Q)$, $x_e\in H^1(Q)$, and the following its bounds hold
\[
 K_0^{-1}(N)\leq\eta\leq K_0(N),\,\ \|u\|_{V_2(Q)}+\|u^2\|_{V_2(Q)}+\|\th\|_{V_2(Q)}
 +\|x_e\|_{H^1(Q)}
 \leq K(N).
\]
In addition, if $N^{-1}\leq\th^0$ on $\Om$, then $K(N)^{-1}\leq\th$ on $Q$.
See \cite{AZRM97} for the existence of a more weak solution and \cite[Proposition 6.2]{AZRM99} for the additional regularity of $u$ and $\th$ presented here as well as see the review \cite{ZSprin17}; the more detailed results on weak solutions can be found there.
The solution is also unique if the first condition \eqref{ggff} holds with $q_e=2$ \cite[Theorem 1.2]{ZASMJ97}, but the next theorem shows that the restriction $q_e=2$ is not required here.
Not all of the listed properties of the solution are used below in the main Theorem~\ref{th3}.

\par We recall that this solution satisfies equations $D_t\eta=Du$ and $D_tx_e=u$
in $L^2(Q)$ together with the initial conditions $\eta|_{t=0}=\eta^0$ and $x_e|_{t=0}=x_e^0$ in $C(0,T;L^2(\Om))$.
Next, it also satisfies the momentum balance equation \eqref{eq2}, where $\s$ is given in \eqref{eq5}, together with the initial condition $u|_{t=0}=u^0$ and partially the boundary conditions \eqref{bc}, in the weak sense
\begin{gather*}
\int_Q\big[-uD_t\vp+(\nu\rho Du-p)D\vp-g[x_e]\vp\big]\,dxdt
 =\int_\Om u^0\vp|_{t=0}\,dx
 -\int_0^T(p_X\vp|_{x=X}-p_0\vp|_{x=0})\,dt,
\end{gather*}
for the same $\vp$ as in the similar integral identity \eqref{weaksolv2}, and also the boundary conditions $u|_{x=0,X}=\mathbf{u}_b$
($m=1$) or $u|_{x=X}=u_X$ ($m=2$) hold in the sense of traces.
Finally, the solution  satisfies the internal energy balance equation \eqref{eq3}, with $\pi$ given in \eqref{eq5}, together with  the initial condition $\th|_{t=0}=\th^0$ and the boundary conditions \eqref{bc4},
in the similar weak sense
\begin{gather}
\int_Q\big[-c_V\th D_t\vp+\la\rho(D\th)D\vp-(\s Du+f[x_e])\vp\big]\,dxdt
\nonumber\\
=\ell(\th^0,\bm{\pi}_b;\vp):=\int_\Om \th^0\vp|_{t=0}\,dx
+\int_0^T(\pi_X\vp|_{x=X}-\pi_0\vp|_{x=0})\,dt,
\label{weaksolv2 for th}
\end{gather}
for any $\vp\in H^1(Q)\cap L^\infty(Q)$, $\vp|_{t=T}=0$; the assumption $\vp\in  L^\infty(Q)$ is required since $\sigma Du\in L^1(Q)$ only.
In fact, this means that $u$ and $\th$ are considered as the weak solutions from $V_2(Q)$ from the corresponding parabolic IBVPs like $\mathcal{L}_m$ and $\mathcal{L}_3$, respectively.

\par But below instead of identity \eqref{weaksolv2 for th} we deal with the weaker  $L^2(Q)$--solution such that
\begin{gather}
\int_Q\big[-\th(c_V D_t\vp+D(\la\rho D\vp)]-(\s Du+f[x_e])\vp\big]\,dxdt
=\ell(\th^0,\bm{\pi}_b;\vp)
\label{weaksolv3 for th}
\end{gather}
for any $\vp\in  H^{2,1;\vk,3}(Q)$,
see identity \eqref{idenL2} for the IBVP $\mathcal{L}_m$ in the case $m=3$.
\par In addition, owing to the (coarsened) well-known $L^\infty(Q)$--bound for the solution from $V_2(Q)$ to the linear parabolic IBVP $\mathcal{L}_m$ we can guarantee that  
\be
 \|u\|_{L^\infty(Q)}
 \leq K(N)\big(\|u^0\|_{L^\infty(\Om)}+\|\mathbf{u}_b\|_{L^\infty(0,T)}+\|\mathbf{p}_b\|_{L^\infty(0,T)}
 +\|p\|_{L^4(Q)}
+\|\bar{g}\|_{L^2(Q)}\big)
 \leq K_1(N),
\label{Linf for u}
\ee
see \cite[Section III.7]{LSU68} and more specifically \cite[Theorem 3.2]{AZMB96}.

\par To analyze the continuous dependence on data in a general framework, along with the stated IBVP $\mathcal{P}_m$, we consider the similar perturbed system of equations
\begin{gather}
D_t\het=D\hu+\beta,
\label{eq1h}\\[1mm]
D_t\hu=D\hsi+\hat{g}[\hx_e],
\label{eq2h}\\[1mm]
c_VD_t\hth=D\hpi+\hsi D\hu+\hat{f}[\hx_e],
\label{eq3h}\\[1mm]
 D_t\hx_e=\hu,\ \ \hsi=\nu\hrho (D\hu+\beta)-\hp,\ \ 
 \hrho=\tfrac{1}{\het},\ \
  \hp=k\hrho\hth,\ \
 \hpi=\la\hrho(D\hth+\g)
\label{eq5h}
\end{gather}
in $Q$, supplemented with the perturbed initial and boundary conditions
\begin{gather}
(\het,\hu,\hth,\hx_e)|_{t=0}=(\het^0,\hu^0,\hth^0,\hx_e^0),
\label{ich}\\[1mm]
\hu|_{x=0,X}=\mathbf{\hu}_b\ (m=1), \
\hsi|_{x=0}=-\hp_0,\ \hu|_{x=X}=\hu_X\ (m=2),\ \
\hsi|_{x=0,X}=-\mathbf{\hp}_b\ (m=3),
\label{bch}\\[1mm]
\hpi|_{x=0,X}={\bm{\hpi}}_b.
\label{bc4h}
\end{gather}
Here $\hx_e^0:=I\hat{\eta}^0+\beta_e$ with some $\beta_e$, as well as $\mathbf{\hu}_b=(\hu_0,\hu_X)$, $\mathbf{\hp}_b=(\hp_0,\hp_X)$ and ${\bm{\hpi}}_b=(\hpi_0,\hpi_X)$ are the pairs of boundary data; once again, 
the functions $\hu_0,\hu_X,\hp_0$ and $\hp_X$ that are missing in the $m$th boundary condition of \eqref{bch} are set equal to $0$.
We denote the perturbed IBVP (\ref{eq1h})--(\ref{bc4h}) by $\widehat{\mathcal{P}}_m$.
In the case $\beta=\gamma=0$ and $\beta_e=0$, the IBVP $\widehat{\mathcal{P}}_m$ is simply another version of $\mathcal{P}_m$.

\par We impose the following conditions on the data of problem $\widehat{\mathcal{P}}_m$. 
\par $(\hat{C}_1)$ Let $N^{-1}\leq\het^0\leq N$ on $\Om$, $\het^0,\hu^0,\beta_e\in L^\infty(\Om)$ and $\hth^0\in L^1(\Om)$.

\par $(\hat{C}_2)$ Let $\hat{g}$ and $\hat{f}$
be measurable on $\mathbb{R}\times Q$, satisfy there the bounds 
$|\hat{g}(\chi,x,t)|\leq\bar{g}(x,t)$ and 
$|\hat{f}(\chi,x,t)|\leq\bar{f}(x,t)$
with $\|\bar{g}\|_{L^1(Q)}\leq N$ and $\bar{f}\in L^1(Q)$,
and $\hat{g}$ and $\hat{f}$ be continuous in $\chi\in\mathbb{R}$ for almost all $(x,t)\in Q$.

\par Let also $\beta\in L^2(Q)$, $\g\in L^1(Q)$ and $\|\g\|_{[V_{2;3_*}(Q)]^*}<\infty$.

\par $(\hat{C}_3)$ Let
$\|\mathbf{\hu}_b\|_{H^1(0,T)}+\|\mathbf{\hp}_b\|_{W^{1,1}(0,T)}\leq N$, $\bm{\hpi}_b\in L^1(0,T)$.

\par Note that these conditions are less restrictive than the above conditions $(C_1)$--$(C_3)$ on the similar data except for the condition on $\mathbf{\hu}_b$.

\par Also we assume that the solution to  problem $\widehat{\mathcal{P}}_m$ has the properties $\het,1/\het,\hx_e\in L^\infty(Q)$,
$\hth,D_t\hat{\eta}, D_t\hx_e$ $\in L^2(Q)$, $\hu\in V_2(Q)\cap L^\infty(Q)$ and satisfies the bound
\[
 K_0^{-1}(N)\leq\het\leq K_0(N),\ \    \|\hu\|_{L^\infty(Q)}+\|D\hu\|_{L^2(Q)}
+\|\hth\|_{L^{\infty,1}(Q)}
+\|\hx_e\|_{L^\infty(Q)}\leq K_1(N).
\]
These assumptions are also less restrictive than the above listed properties of the weak solution to problem $\mathcal{P}_m$. 

\par Also the equations, initial and boundary conditions in the  IBVP $\widehat{\mathcal{P}}_m$ are valid in the similar sense as above for the basic IBVP $\mathcal{P}_m$; in particular, concerning the internal energy balance equation \eqref{eq3h}, with $\hpi$ given in \eqref{eq5h}, the initial condition $\hth|_{t=0}=\hth^0$ and the boundary conditions \eqref{bc4h}, we assume that the integral identity similar to \eqref{weaksolv3 for th} holds, for $\hth$ as $L^2(Q)$-solution of the corresponding $\mathcal{L}_3$-type problem.

\par We define the total  energy, the total initial energy and the modified one 
\[
\hat{e}=\tfrac12\hu^2+c_V\hth,\ \
\hat{e}^0=\tfrac12(\hu^0)^2+c_V\hth^0,\ \ \hat{E}^0:=\half(\hw^0)^2+c_V\hth^0
\]
with $\hw^0:=\hu^0-\hu_\G^0$ and
\be
 \hu_\G^0=\tfrac{1}{\hat{V}^0}[(I^*\hat{\eta}^0)\hu_0(0)+(I\hat{\eta}^0)\hu_X(0)\big]\,\ (m=1),\ \
 \hu_\G^0=\hu_X(0)\,\ (m=2),\ \
 \hu_\G^0=0\,\ (m=3), 
\label{hE0}
\ee
where $\hat{V}^0=\int_\Omega\hat{\eta}^0\,dx$.
The functions $e$, $e^0$, $E^0$, $w^0$ and $u_\G^0$ for problem $\mathcal{P}_m$ are defined similarly (by omitting all the hats).

\par The next theorem on the Lipschitz continuity on the data of the weak solution is the first main result of the paper.
recall that $q_e\in [2,\infty]$ was defined in condition $(C_2)$.
Let $M_1\subset M_0$ be the set of $(q,r)$ such that $q,r\in [2,\infty]$ and $(2q)^{-1}+r^{-1}\leq 1/2$.
\begin{theorem}
\label{th3}
Let the decompositions $\beta=\beta_1+\beta_2$ and
$\hg[\hx_e]-g[\hx_e]=g_1+g_2$ be valid with some $\beta_1,\beta_2,g_1,g_2\in L^1(Q)$ such that the norms of $\beta_1,\beta_2$ and $g_2$ are finite in the next bound.

\par Then the following bound for the difference of weak solutions to problems $\mathcal{P}_m$ and $\widehat{\mathcal{P}}_m$ holds
\begin{gather}
\|\hat{\eta}-\eta\|_{C(0,T;L^2(\Om))}
 +\|\hu-u\|_{L^2(Q)\cap L^\infty(0,T;H^{-1;m})}
 +\|\hat{\theta}-\theta\|_{L^2(Q)}
 +\|\hx_e-x_e\|_{L^{q_e,\infty}(Q)}
\nonumber\\[1mm]
+\|I_t(\hsi-\sigma)\|_{C(0,T;L^2(\Om))}
\leq K(N)\Big(\|\hat{\eta}^0-\eta^0\|_{L^2(\Om)}+
 \|\hu^0-u^0\|_{H^{-1;m}}
 +\|\hat{E}^0-E^0\|_{H^{-1;3}}
\nonumber\\[1mm]
+\|\beta_e\|_{L^{q_e}(\Om)}
+\|\mathbf{\hu}_b-\mathbf{u}_b\|_{W^{1,1}(0,T)}
+\|\mathbf{\hp}_b-\mathbf{p}_b\|_{L^1(0,T)}
+\|\bm{\hpi}_b-\bm{\pi}_b\|_{L^1(0,T)}
\nonumber\\[1mm]
 +\inf_{M_0}\|\beta_1\|_{L^{q,r}(Q)}
 +\inf_{M_1} 
 \|I^{\lan 3\ran}\beta_2\|_{L^{q,r}(Q)}
+\|I_t I^{\lan 1\ran}\beta_2\|_{L^{q_e,\infty}(Q)}
+\|\lan\beta_2\ran_\Omega\|_{L^1(0,T)}
+\|\g\|_{[V_{2;3_*}(Q)]^*}
\nonumber\\
 +\|g_1\|_{L^1(Q)}
 +\|I^{\lan m\ran}g_2\|_{L^2(Q)}+\de_{m3}\|\lan g_2\ran_\Omega\|_{L^1(0,T)}
 +\|\hf[\hx_e]-f[\hx_e]\|_{[H^{2,1;\hrho,m}(Q)]^*}\Big).
\label{hzz}
\end{gather}
\end{theorem}
\begin{rem}
\label{rem1}
1. We have $\hat{E}^0-E^0=\hat{e}^0-e^0$ in the case $m=3$.
The term $\hat{E}^0-E^0$ can be simplified down to $\hat{e}^0-e^0$  in the cases $m=1,2$ as well provided that the norm $\|\hu^0-u^0\|_{H^{-1;m}}$ is coarsened as $\|\hu^0-u^0\|_{H^{-1;3}}$.

\smallskip\par 2. The term $\|I_t I^{\lan 1\ran}\beta_2\|_{L^{q_e,\infty}(Q)}$ can be omitted for $q_e=2$ or restricting the values $q\in [q_e,\infty]$ in the definition of $M_1$.

\smallskip\par 3. If the additional bound $\|D\hu\|_{L^{2,\infty}(Q)}\leq \hat{K}(N)$ is valid, then the norms of $I^{\lan 3\ran}\beta_2$ and $I^{\lan m\ran}g_2$ can be weakened down to
$\|I^{\lan 3\ran}\beta_2\|_{L^2(Q)}$ and $\|I^{\lan m\ran}g_2\|_{L^{2,1}(Q)}$.
\end{rem}
\begin{proof}
The proof comprises five steps.
At each step, we sequentially derive bounds for the terms on the left-hand side of \eqref{hzz} by different special means.

\smallskip\par 1. We first bound $\hat{\eta}-\eta$.
To this end, we derive more suitable equations for $\hat{\eta}$ and $\eta$. We multiply equation $D_t\het=D\hu+\beta$ by $\nu\hrho$ and rewrite  it as
$\nu D_t\ln\hat{\eta}=\hsi+\hp$.
Integrating in $t$ leads to the formula
\be
\nu\ln\hat{\eta}=\nu\ln\hat{\eta}^0+I_t\hsi+I_t\hp.
\label{lne}
\ee
\par We consider the $\mathcal{L}_m$-type IBVP for $\hu$ formed by equation \eqref{eq2h}, the initial condition $\hu|_{t=0}=\hu^0$ and the boundary conditions \eqref{bch}. 
Due to \cite[Lemma 2.1]{AZMN94}, there exists $DI_t\hsi\in L^{1,\infty}(Q)$, and the following equalities with $I_t\hsi$ hold
\begin{gather}
\hu=\hu^0+DI_t\hsi+I_t\hg[\hx_e],
\label{dits}\\[1mm]
I_t\hsi|_{x=0}=-I_t\hp_0\ \ (m=2,3),\ \
 I_t\hsi|_{x=X}=-I_t\hp_X\ \ (m=3).
\label{itsg}
\end{gather}
We apply the operator $I^{\lan m\ran}$ to equality (\ref{dits}) and owing to formulas (\ref{idp}) and (\ref{itsg}) get the representations
\begin{gather}
I_t\hsi=I^{\lan 1\ran}(\hu-\hu^0-I_t\hg[\hx_e])+I_t\lan\hsi\ran_\Omega\ \ (m=1),
\label{its1}\\[1mm]
I_t\hsi=I^{\lan m\ran}(\hu-\hu^0-I_t\hg[\hx_e])+I_t\hsi_\G\ \ (m=2,3)
\label{its2}
\end{gather}
(see them also in \cite[Lemma 3.1]{ZASMJ97}), where
\be
 \hsi_\G=0\ \ (m=1),\ \
 \hsi_\G=-\hp_0\ \ (m=2),\ \
 \hsi_\G=-(1-\tfrac{x}{X})\hp_0-\tfrac{x}{X}\hp_X\ \ (m=3).
\label{si G}
\ee
The function $\sigma_\G$ for problem $\mathcal{P}_m$ is defined similarly.
We insert formula \eqref{its2} into (\ref{lne}) and get the equation
\be
 \nu\ln\hat{\eta}=I^{\lan m\ran}\hu+I_t\hp
 +\nu\ln\hat{\eta}^0-I^{\lan m\ran}\hu^0
 +I_t\hsi_\G-I^{\lan m\ran}I_t\hg[\hx_e]\ \ (m=2,3).
\label{lne23}
\ee
\par But, in the case $m=1$, the corresponding equation does not lead to the desired result, and
we are forced to proceed in a more complicated manner.
We subtract from equality (\ref{lne}) the same equality taken at $x=\zeta\in\Om$ and
divide the difference by $\nu$.
Applying the function $\exp$ to the result, we find
\[
\hat{\eta}\hat{\eta}^0(\zeta)=\hat{\eta}|_{x=\zeta}\hat{\eta}^0
 \exp\big[\tfrac{1}{\nu}\big(I_t\hsi-I_t\hsi|_{x=\zeta}+I_t\hp-I_t\hp|_{x=\zeta}\big)\big].
\]
Expressing the difference of values of $I_t\hsi$ by virtue of formula (\ref{its1}), we obtain
\[
\hat{\eta}\hat{\eta}^0(\zeta)e^{\hA|_{x=\zeta}}
=\hat{\eta}|_{x=\zeta}\hat{\eta}^0e^{\hA},\ \ \text{with}\ \
 \hA:=\tfrac{1}{\nu}\big[I_t\hp+I^{\lan 1\ran}(\hu-\hu^0-I_t\hg[\hx_e])\big].
\]
Finally, taking the mean value in $\zeta$ of the both parts, we derive the more suitable equation 
\be
 \hat{\eta}\lan\hat{\eta}^0e^{\hA}\ran_\Om
 =\lan\hat{\eta}\ran_\Omega\hat{\eta}^0e^{\hA}
\label{lne1}
\ee
(see it also in \cite[Lemma 1]{ZAMN98}).
In addition, applying the operator $I_t\lan\cdot\ran_\Om$ to the equation $D_t\het=D\hu+\beta$ and using the boundary conditions (\ref{bch}) for $m=1$, we find the formula for the mean gas volume
\be
 \lan\hat{\eta}\ran_\Om=\lan\hat{\eta}^0\ran_\Om
 +X^{-1}I_t(\hu_X-\hu_0)+I_t\lan\beta\ran_\Omega\ \ (m=1)
\label{mheta}
\ee
similar to formula (\ref{meta}) for problem $\mathcal{P}_1$.

\par For $m=1,2,3$, we are ready to prove the following bound
\begin{gather}
\|\hat{\eta}-\eta\|_{C(0,t;L^2(\Om))}
 \leq  K^{(0)}(N,T)\big(\|\hu-u\|_{L^\infty(0,t;H^{-1;m})}
 +\|\hth-\theta\|_{L^{2,1}(Q_t)}
\nonumber\\[1mm]
+\|\hat{\eta}^0-\eta^0\|_{L^2(\Om)}
+\|\hu^0-u^0\|_{H^{-1;m}}
+\de_{m1}\|I_t(\mathbf{\hu}_b-\mathbf{u}_b)\|_{L^\infty(0,T)}
+\|I_t(\mathbf{\hp}_b-\mathbf{p}_b)\|_{L^\infty(0,T)}
\nonumber\\[1mm]
+\de_{m1}\|I_t\lan\beta\ran_\Omega\|_{L^\infty(0,T)}
+\|I^{\lan m\ran}I_t(\hg[\hx_e]-g[x_e])\|_{L^{2,\infty}(Q_t)}
\big),\ \ 0<t\leq T.
\label{hee1}
\end{gather}

\par In the cases $m=2,3$, we subtract from equation (\ref{lne23}) the similar one for problem $\mathcal{P}_m$, and due to the two-sided uniform bounds $K_0^{-1}\leq\eta\leq K_0$ and $K_0^{-1}\leq\hat{\eta}\leq K_0$ we get
\begin{gather}
\nu K_0^{-1}|\hat{\eta}-\eta|\leq |I^{\lan m\ran}(\hu-u)|+I_t|\hp-p|
 +\nu K_0|\hat{\eta}^0-\eta^0|+|I^{\lan m\ran}(\hu^0-u^0)|
\nonumber\\[1mm]
+|I_t(\hsi_\G-\sigma_\G)|+|I^{\lan m\ran}I_t(\hg[\hx_e]-g[x_e])|.
\label{hee2}
\end{gather}
Clearly
\be
 \hp-p=k(\hrho-\rho)\theta+k\hrho(\hat{\theta}-\theta)
 =-k\hrho\rho(\hat{\eta}-\eta)\theta+k\hrho(\hat{\theta}-\theta).
\label{hpp}
\ee
Therefore, taking the $L^2(\Om)$--norm of the both sides of inequality (\ref{hee2}), we obtain
\begin{gather}
\hspace{-8pt}\nu K_0^{-1}\|\hat{\eta}-\eta\|_{L^2(\Om)}
 \leq kK_0^2I_t\big(\|\theta\|_{L^\infty(\Om)}\|\hat{\eta}-\eta\|_{L^2(\Om)}\big)+
 \|\hu-u\|_{H^{-1;m}}
 +kK_0I_t\|\hth-\theta\|_{L^2(\Om)}
\nonumber\\[1mm]
+\nu K_0\|\hat{\eta}^0-\eta^0\|_{L^2(\Om)}+\|\hu^0-u^0\|_{H^{-1;m}}
 + \|I_t(\hsi_\G-\sigma_\G)\|_{L^2(\Om)}
 +\|I^{\lan m\ran}I_t(\hg[\hx_e]-g[x_e])\|_{L^2(\Om)}
\label{hee3}
\end{gather}
on $(0,T)$.
Applying the Gronwall--Bellman lemma and the bound $\|\theta\|_{L^{\infty,1}(Q)}\leq K$,
we prove bound (\ref{hee1}).

\par In the case $m=1$, we subtract from equation (\ref{lne1}) the similar one for problem $\mathcal{P}_m$ and get
\[
 (\hat{\eta}-\eta)\lan\eta^0e^A\ran_\Om
 =-\hat{\eta}\lan\hat{\eta}^0e^{\hA}-\eta^0e^A\ran_\Om
 +\lan\hat{\eta}\ran_\Omega\big(\hat{\eta}^0e^{\hA}-\eta^0e^A\big)
 +\big(\lan\hat{\eta}\ran_\Om-\lan\eta\ran_\Omega\big)\eta^0e^A.
\]
Obviously also
$\hat{\eta}^0e^{\hA}-\eta^0e^A
=\hat{\eta}^0\big(e^{\hA}-e^A\big)+(\hat{\eta}^0-\eta^0)e^A$.
We apply the two-sided uniform bounds $K_0^{-1}\leq\eta\leq K_0$ and $K_0^{-1}\leq\hat{\eta}\leq K_0$, the bound
\[
 |A|\leq\tfrac{1}{\nu}\big(kK_0\|\theta\|_{L^{\infty,1}(Q)}
 +\|u\|_{L^{1,\infty}(Q)}+\|u^0\|_{L^1(\Om)}+\|\bar{g}\|_{L^1(Q)}\big)
\leq K
\]
and the similar bound $|\hA|\leq K$. 
Then we get
\[
 |\hat{\eta}-\eta|
 \leq K\big(\|\hA-A\|_{L^1(\Om)}+\|\hat{\eta}^0-\eta^0\|_{L^1(\Om)}
 +|\hA-A|+|\hat{\eta}^0-\eta^0|+|\lan\hat{\eta}\ran_\Om-\lan\eta\ran_\Om|\big).
\]
Taking the $L^2(\Om)$--norm of the both sides, we derive
\[
 \|\hat{\eta}-\eta\|_{L^2(\Om)}
 \leq K(\|\hA-A\|_{L^2(\Om)}+\|\hat{\eta}^0-\eta^0\|_{L^2(\Om)}
 +|\lan\hat{\eta}\ran_\Om-\lan\eta\ran_\Om|).
\]
Clearly further we have
\begin{gather*}
|\hA-A|\leq\tfrac{1}{\nu}\big(I_t|\hp-p|
 +|I^{\lan 1\ran}(\hu-u)|+|I^{\lan 1\ran}(\hu^0-u^0)|
 +|I^{\lan 1\ran}I_t(\hg[\hx_e]-g[x_e])|\big),
\\[1mm]
|\lan\hat{\eta}\ran_\Om-\lan\eta\ran_\Om|
 \leq X^{-1}\|\hat{\eta}^0-\eta^0\|_{L^1(\Om)}
 +X^{-1}|I_t[\hu_0-u_0-(\hu_X-u_X)]|+|I_t\lan\beta\ran_\Om|,
\end{gather*}
see formulas (\ref{meta}) and (\ref{mheta}).
Therefore, similarly to (\ref{hee3}) we obtain
\begin{gather*}
\|\hat{\eta}-\eta\|_{L^2(\Om)}
 \leq K\big[I_t\big(\|\theta\|_{L^\infty(\Om)}\|\hat{\eta}-\eta\|_{L^2(\Om)}\big)+
 \|\hu-u\|_{H^{-1;1}}
 +I_t\|\hth-\theta\|_{L^2(\Om)}
 +\|\hat{\eta}^0-\eta^0\|_{L^2(\Om)}
\\[1mm]
 +\|\hu^0-u^0\|_{H^{-1;m}}
 +|I_t[\hu_0-u_0-(\hu_X-u_X)]|+|I_t\lan\beta\ran_\Om|
 +\|I^{\lan 1\ran}I_t(\hg[\hx_e]-g[x_e])\|_{L^2(\Om)}\big]
\end{gather*}
on $(0,T)$.
Clearly this inequality implies
bound (\ref{hee1}) for $m=1$ as well.
 
\smallskip\par 2. Next we bound $v=\hu-u$.
This function is the solution from $V_2(Q)$ to the $\mathcal{L}_m$--type IBVP with $\vk=\nu\hrho$ for the equation
\[
 D_tv=Ds+\hg[\hx_e]-g[x_e],\ \ s:=\hsi-\sigma=\nu\hrho Dv+\psi
\]
in $Q$, with the initial condition $v|_{t=0}=\hu^0-u^0$ and the boundary conditions
\begin{gather*}
v|_{x=0,X}=\mathbf{\hu}_b-\mathbf{u}_b\ (m=1), \
s|_{x=0}=-(\hp_0-p_0),\ v|_{x=X}=\hu_X-u_X\ (m=2),
\\
s|_{x=0,X}=-(\mathbf{\hp}_b-\mathbf{p}_b)\ (m=3).
\end{gather*}
We derive the following simple formula with $\psi$:
\[
\hat{\eta}\psi=\hat{\eta}[\hsi-\sigma-\nu\hrho D(\hu-u)]=\hat{\eta}\hsi-\eta\sigma -(\hat{\eta}-\eta)\sigma-\nu D(\hu-u)
=\nu\beta-k(\hat{\theta}-\theta)-(\hat{\eta}-\eta)\sigma.
\]
Applying Proposition \ref{th1} to this problem in $Q_t$, we get the bound
\begin{gather}
\|\hu-u\|_{L^2(Q_t)\cap L^\infty(0,t;H^{-1;m})}
 \leq K(N,T)\big(\|\hu^0-u^0\|_{H^{-1;m}}
\nonumber\\[1mm]
+\|\mathbf{\hu}_b-\mathbf{u}_b\|_{L^{4/3}(0,T)}
+\|\mathbf{\hp}_b-\mathbf{p}_b\|_{L^1(0,T)}
+\|\hat{\eta}\psi\|_{[V_{2;\hrho,m_*}(Q_t)]^*}
\nonumber\\[1mm]
+\|\hat{\eta} I^{\lan m\ran}(\hg[\hx_e]-g[x_e])\|_{[V_{2;\hrho,m_*}(Q_t)]^*}
+\de_{m3}I_t|\lan \hg[\hx_e]-g[x_e]\ran_\Om|
\big),\ \ 0<t\leq T.
\label{huu1}
\end{gather}
Due to inequality (\ref{ineq3}) we find
\begin{gather}
\|\hat{\eta}\psi\|_{[V_{2;\hrho,m_*}(Q_t)]^*}
\leq C\big(\|(\hat{\eta}-\eta)\sigma\|_{L^{1,4/3}(Q_t)}
 +k\|\hat{\theta}-\theta\|_{L^{2,1}(Q_t)}
 +\nu\|\beta\|_{[V_{2;\hrho,m_*}(Q_t)]^*}\big)
\nonumber\\[1mm]
\leq K_1(\|\hat{\eta}-\eta\|_{L^{2,4}(Q_t)}+\|\hat{\theta}-\theta\|_{L^{2,1}(Q_t)}
+\|\beta\|_{[V_{2;\hrho,m_*}(Q_t)]^*}\big)
\label{psi1}
\end{gather}
since $\|(\hat{\eta}-\eta)\sigma\|_{L^{1,4/3}(Q_t)}
\leq\|\sigma\|_{L^2(Q)}\|\hat{\eta}-\eta\|_{L^{2,4}(Q_t)}$ and $\|\sigma\|_{L^2(Q)}\leq K$.
Also, for $z\in L^1(Q)$, using inequality \eqref{ineq3} once again, we find
\be
 \|I^{\lan m\ran}I_tz\|_{L^{2,\infty}(Q_t)}+\|\hat{\eta} I^{\lan m\ran}z\|_{[V_{2;\hrho,m_*}(Q_t)]^*}
 \leq (1+CK_0)\|I^{\lan m\ran}z\|_{L^{2,1}(Q_t)}.
\label{imitw}
\ee
\par We apply bound \eqref{psi1} in (\ref{huu1}).
Then we multiply the result by $2K^{(0)}$, add it to bound (\ref{hee1}) and, applying bound \eqref{imitw} for $z=\hg[\hx_e]-g[x_e]$ and coarsening some data norms, we derive
\begin{gather}
\|\hat{\eta}-\eta\|_{C(0,t;L^2(\Om))}
 +\|\hu-u\|_{L^2(Q_t)\cap L^\infty(0,t;H^{-1;m})}
 \leq K^{(1)}(N,T)\big(\|\hat{\eta}-\eta\|_{L^{2,4}(Q_t)}
 +\|\hat{\theta}-\theta\|_{L^{2,1}(Q_t)}
\nonumber\\[1mm]
+\|\hat{\eta}^0-\eta^0\|_{L^2(\Om)}
 +\|\hu^0-u^0\|_{H^{-1;m}}
 +\|\mathbf{\hu}_b-\mathbf{u}_b\|_{L^{4/3}(0,T)}+\|\mathbf{\hp}_b-\mathbf{p}_b\|_{L^1(0,T)}
+\|\beta\|_{[V_{2;\hrho,m_*}(Q)]^*}
\nonumber\\[1mm]
+\de_{m1}\|I_t\lan\beta\ran_\Omega\|_{L^\infty(0,T)}
 +\|I^{\lan m\ran}(\hg[\hx_e]-g[x_e])\|_{L^{2,1}(Q_t)}
 +\de_{m3}I_t|\lan \hg[\hx_e]-g[x_e]\ran_\Om|\big),\ 0<t\leq T.
\label{huu2}
\end{gather}
The presence of the space $L^\infty(0,t;H^{-1;m})$ on the left in bound \eqref{huu1} has been essential here to remove the term $\hu-u$ from the right-hand side of \eqref{hee1}.
\par Owing to Lemma \ref{lem1} (applied on $(0,t)$ in the role of $(0,T)$) this bound implies the improved one without the term $\|\hat{\eta}-\eta\|_{L^{2,4}(Q_t)}$ on the right; we do not write it down for brevity.

\smallskip\par 3. It seems impossible to bound $\|\hat{\theta}-\theta\|_{L^2(Q_t)}$ properly by virtue of the internal energy balance equations (\ref{eq3}) and (\ref{eq3h}) because of the terms $\s Du$ and $\hsi D\hu$. Thus we first turn to a bound for $\hat{e}-e$.
But under the inhomogeneous boundary conditions \eqref{bc} and \eqref{bch} in the cases  $m=1,2$ one cannot pose boundary conditions for this function.
Therefore, we need to remove these inhomogeneities.
This step is the most cumbersome in the proof (and completely new compared to the barotropic case \cite{AZMN94}).
\par We define the function $\hw:=\hu-\hu_\G$ with
\[
 \hu_\G=:\tfrac{1}{\hV}[(I^*\het)\hu_0+(I\het)\hu_X]=\hu_0+\tfrac{1}{\hV}(I\het)(\hu_X-\hu_0)\ (m=1),\,\
 \hu_\G=\hu_X\ (m=2),\,\
 \hu_\G=0\ (m=3).
\]
The functions $u_\G$ and $w:=u-u_\G$ for problem $\mathcal{P}_m$ are defined similarly.
In the case $m=1$, below we need a short formula for $D_t\hu_\G$.  Using the equation
$D_t\het=D\hu+\beta$, we derive
\begin{gather}
D_t\hu_\G=(D_t\hu)_\G+\tfrac{1}{\hV}(\hu-\hu_0+I\beta)\hu_{X0}
-\tfrac{1}{\hV^2}\Big(\hu_{X0}+\int_\Om\beta\,dx\Big)(I\het)\hu_{X0}
\nonumber\\[1mm]
=(D_t\hu)_\G+\tfrac{1}{\hV}\big(\hu-\hu_\G+I\beta-\tfrac{X}{\hat{V}}(I\het)\lan\beta\ran_\Om\big)\hu_{X0}
 =(D_t\hu)_\G+\tfrac{1}{\hV}(\hw+IP_{\het}\beta)\hu_{X0}\ \ (m=1),
\label{dthuG}
\end{gather}
where $(D_t\hu)_\G:=D_t\hu_0+\tfrac{1}{\hV}(I\het)D_t(\hu_X-\hu_0)$ and $\hu_{X0}:=\hu_X-\hu_0$;
recall also definition \eqref{pvk} of $P_{\hat{\eta}}$.
We also set $(D_t\hu)_\G:=D_t\hu_X$ ($m=2$) and $(D_t\hu)_\G:=0$ ($m=3$).
The functions $(D_tu)_\G$ and $u_{X0}$ for problem $\mathcal{P}_m$ are defined similarly.
 
\par We begin with rewriting equation $D_t\hu=D\hsi+\hat{g}[\hx_e]$ with respect to $\hw$ in the form
\be
D_t\hw=DS_{\hw}+F_{\hw}\ \ \text{in}\ \ Q,
\label{eq2hw}
\ee
with the flux and the free term
\be
S_{\hw}:=\hsi-\nu\hrho D\hu_\G-\hsi_\G=\nu\hrho D\hw-\hp+\nu\hrho\beta-\hsi_\G,\ \ 
F_{\hw}:=\hat{g}[\hx_e]-D_t\hu_\G+D\hsi_\G,
\label{SwFw}
\ee
with $\hsi_\G$ defined in \eqref{si G}, where the independence of $\hrho D\hu_\G$ on $x$ has been essential:
\be
 \hrho D\hu_\G=\tfrac{1}{\hat{V}}(\hu_X-\hu_0)\ \ (m=1),\ \
 D\hu_\G=0\ \ (m=2,3).
\label{DhuG}
\ee
Now the boundary conditions \eqref{bch} can be transformed to the following homogeneous form
\begin{gather}
\hw|_{x=0,X}=0\ \ (m=1);\ \
S_{\hw}|_{x=0}=0,\,\hw|_{x=X}=0\,\ (m=2);\ \
S_{\hw}|_{x=0,X}=0\ \ (m=3).
\label{bc3hw}
\end{gather}
More precisely, we mean that $\hw$ is the weak solution from $V_2(Q)$ to the ``linear'' parabolic IBVP of $\mathcal{L}_m$--type formed by equation \eqref{eq2hw}, the initial condition $\hw|_{t=0}=\hw^0:=\hu^0-\hu_\G^0$ (recall formulas \eqref{hE0} for $\hu_\G^0$)
and the boundary conditions \eqref{bc3hw}.
Note that also $\|\hw\|_{L^\infty(Q)}\leq K$.

\par For $\hw$, the additional integral identity holds (which is crucial for us)
\begin{gather*}
 \int_Q\big[-\half\hw^2D_ty+S_{\hw}\big((D\hw)y+\hw Dy\big)\big]\,dxdt
=\int_\Omega\half(\hw^0)^2y|_{t=0}\,dx+\int_QF_{\hw}\hw y\,dxdt
\end{gather*}
for any $y\in H^1(Q)\cap L^\infty(Q)$ with $y|_{t=T}=0$, for example, see \cite[Proposition 2.3]{AZMB96} or \cite[Lemma~5.4]{AZRM97}.
Formally it appears after multiplying equation \eqref{eq2hw} by $wy$ and integrating by parts in $x$ and $t$ and using the boundary conditions \eqref{bc3hw}. 
Thus $\hw$ is the weak solution to the IBVP
\begin{gather}
D_t(\half\hw^2)=D(S_{\hw}\hw)-S_{\hw}D\hw+F_{\hw}\hw\ \ \text{in}\ \ Q,\ \
\half\hw^2|_{t=0}=\half(\hw^0)^2,\ \ (S_{\hw}\hw)|_{x=0,X}=0. 
\label{ibvp wh2}
\end{gather}
From both formulas \eqref{SwFw} for $S_{\hw}$ we derive the auxiliary formula
\begin{gather}
S_{\hw}D\hw=(\hsi-\nu\hrho D\hu_\G-\hsi_\G)D\hu-\het S_{\hw}\hrho D\hu_\G
=(\hsi-\hsi_\G)D\hu-(\nu D\hu+\het S_{\hw})\hrho D\hu_\G
\nonumber\\
=\hsi D\hu-\big(2\nu D\hw-k\hth+\nu\beta+\nu D\hu_\G\big)\hrho D\hu_\G-\hsi_\G D\hu,
\label{Sw-sigDu}
\end{gather}
where we have used that $\hsi_\G D\hu_\G=0$ for all $m=1,2,3$.

\par Adding to the relations of the IBVP \eqref{ibvp wh2} for $\hw$ the corresponding relations for $\hth$: 
\[
c_VD_t\hth=D\hpi+\hsi D\hu+\hat{f}[\hx_e]\ \ \text{in}\ \ Q,\ \ c_V\hth|_{t=0}=c_V\hth^0,\ \ 
\hpi|_{x=0,X}={\bm{\hpi}}_b
\]
(actually we add the corresponding  integral identities), 
we derive that the modified total energy $\hat{E}:=\half\hw^2+c_V\hth$ is the solution from $L^2(Q)$ to the ``linear'' parabolic IBVP of $\mathcal{L}_3$-type
\begin{gather*}
D_t\hat{E}=D\hat{S}+\hat{F}+\hsi_\G D\hu\ \ \text{in}\ \ Q,\ \
\hat{E}|_{t=0}=\hat{E}^0,\ \
\hat{S}|_{x=0,X}={\bm{\hpi}}_b,
\end{gather*}
with $\vk=\frac{\la}{c_V}\hrho$, $\hat{E}^0$ defined above just before Theorem \ref{th3} and the following flux and free terms
\begin{gather*}
\hat{S}:=S_{\hw}\hw+\hpi
 =\tfrac{\la}{c_V}\hrho D\hat{E}+\hat{\Psi},\ \
\hat{\Psi}:=
(\nu-\tfrac{\la}{c_V})\hrho\,\half D(\hw^2)
-(\hp-\nu\hrho\beta
+\hsi_\G)\hw+\la\hrho\g,
\\[1mm]
\hat{F}:=
-S_{\hw}D\hw+F_{\hw}\hw+\hsi D\hu+\hat{f}[\hx_e]-\hsi_\G D\hu
\\
=(\hat{g}[\hx_e]-D_t\hu_\G+\delta_{m3}D\hsi_\G)\hw
 +\delta_{m1}(2\nu D\hw-k\hth+\nu\beta+\nu D\hu_\G)\hrho D\hu_\G
 +\hat{f}[\hx_e].
\end{gather*}
Here formulas \eqref{SwFw} for $S_{\hw}$ and $F_{\hw}$ and \eqref{Sw-sigDu} and the properties $D\hsi_\G=0$ ($m=1,2$) and $D\hu_\G=0$ ($m=2,3$) have been applied.

\par Similarly, the function $E:=\half w^2+c_V\theta$ is the solution from $L^2(Q)$ to the IBVP
\begin{gather*}
D_tE=DS+F+\sigma_\G Du\ \ \text{in}\ \ Q,\ \
E|_{t=0}=E^0,\ \
S|_{x=0,X}={\bm{\pi}}_b,
\end{gather*}
with $\vk=\frac{\la}{c_V}\hrho$ once again and the similar terms without hats and for $\beta=\g=0$, namely,
\begin{gather*}
S:=(\nu\rho Dw-p-\sigma_\G)w+\pi
 =\tfrac{\la}{c_V}\hrho DE-\tfrac{\la}{c_V}(\hrho-\rho)DE+\Psi,
\\[1mm]
\Psi:=(\nu-\tfrac{\la}{c_V})\rho\,\half D(w^2)-(p+\sigma_\G)w,
\\[1mm]
F:=(g[x_e]-D_t u_\G+\delta_{m3}D\sigma_\G)w
 +\delta_{m1}(2\nu Dw-k\theta+\nu Du_\G)\rho Du_\G
 +f[x_e].
\end{gather*}

\par We subtract the relations of the IBVPs for $\hat{E}$ and $E$ 
and pass to the corresponding IBVP for $\hat{E}-E$ (in more detail, we subtract the corresponding integral identities for their $L^2(Q)$--solutions). 
We need to rewrite the difference of the last terms on the right in the equations for $\hat{E}$ and $E$ as follows
\begin{gather*}
\hsi_\G D\hu-\sigma_\G Du=(\hsi_\G-\sigma_\G)Du+\hsi_\G(D\hu-Du)
=(\hsi_\G-\sigma_\G)Du-\hsi_\G\b-(D_t\hsi_\G)(\hat{\eta}-\eta)+D_tG
\end{gather*}
with $G:=\hsi_\G(\hat{\eta}-\eta)-\hsi_\G|_{t=0}(\hat{\eta}^0-\eta^0)$,
where the equations $D_t\het=D\hu+\beta$ and $D_t\eta=Du$
have been used.
Now Proposition \ref{th2} in the case $m=3$, applied to the IBVP for $v=\hat{E}-E$ in $Q_t$, implies the bound 
\begin{gather}
\|\hat{E}-E\|_{L^2(Q_t)}
 \leq K(N,T)\big(\|\hat{E}^0-E^0\|_{H^{-1;3}}
 +\|\bm{\hpi}_b-\bm{\pi}_b\|_{L^1(0,T)}
\nonumber\\[1mm]
 +\|\hat{\eta}\big[\tfrac{\la}{c_V}(\hrho-\rho)DE+\hat{\Psi}-\Psi\big]\|_{[V_{2;3_*}(Q_t)]^*}
 +\|\hat{F}-F-\hsi_\G\b\|_{[H^{2,1;\hrho,3}(Q_t)]^*}
\nonumber\\[1mm]
 +\|(\hsi_\G-\sigma_\G)Du
 -(D_t\hsi_\G)(\hat{\eta}-\eta)\|_{[H^{2,1;\hrho,3}(Q_t)]^*}+\|G\|_{L^2(Q_t)}\big),\ \ 0<t\leq T.
\label{hEE1}
\end{gather}

\par We estimate the right-hand side of this inequality term by term and begin with the third one.
We carry out the following transformations
\begin{gather*}
\hat{\eta}\big[\tfrac{\la}{c_V}(\hrho-\rho)DE+\hat{\Psi}-\Psi\big]
\\[1mm]
=\hat{\eta}\big\{(\hrho-\rho)[\tfrac{\la}{c_V}(Dw)w+\la D\theta]
+(\nu-\tfrac{\la}{c_V})\big[(\hrho-\rho)(Dw)w+\hrho\half D(\hw^2-w^2)\big]
\\[1mm]
 -(\hp-p+\hsi_\G-\sigma_\G)w
 -(\hp+\hsi_\G)(\hw-w)
 +\nu\hrho\beta\hw+\la\hrho\g\vphantom{\frac{1}{2}}\big\}
\\[1mm]
=(\nu-\tfrac{\la}{c_V})\half D(\hw^2-w^2)
 -(\hat{\eta}-\eta)\big[(\nu\rho Dw-p)w+\pi\big]-k(\hat{\theta}-\theta)w
\\[1mm]
-\hat{\eta}(\hsi_\G-\sigma_\G)w
 -(k\hat{\theta}
 +\hat{\eta}\hsi_\G)(\hw-w)
 +\nu\beta\hw+\la\g,
\end{gather*}
where formula (\ref{hpp}) for $\hp-p$ has been used once again.
Applying the definition of the norm $\|\cdot\|_{[V_{2;3_*}(Q_t)]^*}$ and inequality (\ref{ineq3}) with it, we derive the bound
\begin{gather*}
 \|\hat{\eta}[\tfrac{\la}{c_V}(\hrho-\rho)DE
 +\hat{\Psi}-\Psi]\|_{[V_{2;3_*}(Q_t)]^*}
 \leq\half|\nu-\tfrac{\la}{c_V}|\|\hw^2-w^2\|_{L^2(Q_t)}
\\[1mm]
+C\|\hat{\eta}-\eta\|_{L^{2,\infty}(Q_t)}\big(\|\sigma\|_{L^{3/2}(Q)}\|w\|_{L^{\infty}(Q)}+\|\pi\|_{L^{3/2}(Q)}\big)
+k\|\hth-\theta\|_{L^{2,1}(Q_t)}\|w\|_{L^{\infty}(Q)}
\\[1mm]
+\|\hat{\eta}\|_{L^\infty(Q)}
 \|\hsi_\G-\sigma_\G\|_{L^{\infty,1}(Q)}\|w\|_{L^{2,\infty}(Q)}
\\[1mm]
+\big(Ck\|\hth\|_{L^3(Q)}
+\|\hat{\eta}\|_{L^\infty(Q)}\|\hsi_\G\|_{L^{\infty,2}(Q)}\big)\|\hw-w\|_{L^2(Q_t)}
+\nu\|\beta\hw\|_{[V_{2;3_*}(Q)]^*}+\la\|\g\|_{[V_{2;3_*}(Q)]^*}.
\end{gather*}

\par Let us estimate the norms of differences on its right-hand side. Firstly we have
\begin{gather*}
\|\hw^2-w^2\|_{L^2(Q_t)}
 \leq\big(\|\hw\|_{L^\infty(Q)} +\|w\|_{L^\infty(Q)}\big)\|\hw-w\|_{L^2(Q_t)}
\\[1mm]
\leq \big(\|\hu\|_{L^\infty(Q)}+\|\mathbf{\hu}_b\|_{L^\infty(0,T)}+\|u\|_{L^\infty(Q)}+\|\mathbf{u}_b\|_{L^\infty(0,T)}\big)
\big(\|\hu-u\|_{L^2(Q_t)}+\|\hu_\G-u_\G\|_{L^2(Q_t)}\big).
\end{gather*}
The bound $\|u\|_{L^\infty(Q)}\leq K_1$, see \eqref{Linf for u}, and the assumption $\|\hu\|_{L^\infty(Q)}\leq K$ are essential here.
In the case $m=1$, further we write down the formulas and bound
\begin{gather}
\hu_\G-u_\G=\tfrac{1}{\hV}(I^*\hat{\eta})(\hu_0-u_0)+\tfrac{1}{\hV}(I\hat{\eta})\big(\hu_X-u_X\big)
+(\tfrac{1}{\hV}I\hat{\eta}-\tfrac{1}{V}I\eta)(u_X-u_0),
\nonumber\\[1mm]
|\tfrac{1}{\hV}I\hat{\eta}-\tfrac{1}{V}I\eta|
 =\tfrac{1}{\hV V}\big|[I(\hat{\eta}-\eta)]I^*\het-[I^*(\het-\eta)]I\het\big|
 \leq\tfrac{1}{V}\|\hat{\eta}-\eta\|_{L^1(\Om)}
\label{hVIeta}
\end{gather}
since $I\hat{\eta}+I^*\hat{\eta}=\hV$ and $I\eta+I^*\eta=V$. They imply the bound
\begin{gather}
\|\hu_\G-u_\G\|_{L^2(Q_t)}
 \leq \sqrt{X}\big(\|\mathbf{\hu}_b-\mathbf{u}_b\|_{L^2(0,T)}
+N\|\mathbf{u}_b\|_{L^2(0,T)}\|\hat{\eta}-\eta\|_{L^{1,\infty}(Q_t)}\big),
\label{hugug3}
\end{gather}
see condition \eqref{gas volume}. 
In the cases $m=2,3$, obviously the simpler bound, without the second term on the right, holds as well.

\par Moreover, clearly the following bound holds
\be
 \|\hsi_\G-\sigma_\G\|_{L^{\infty,1}(Q)}+\|D(\hsi_\G-\sigma_\G)\|_{L^{\infty,1}(Q)}
 \leq C\|\mathbf{\hp}_b-\mathbf{p}_b\|_{L^1(0,T)}.
\label{hsgsg}
\ee
For $\beta=\beta_1+\beta_2$, owing to the identity
\begin{gather}
\int_Q\beta\hw\vp\,dxdt
=\int_Q\big(\beta_1+\lan\beta_2\ran_\Omega\big)\hw\vp\,dxdt
 -\int_Q(I^{\lan 3\ran}\beta_2)\big[(D\hw)\vp+\hw D\vp\big]\,dxdt,
\label{id b w vp}
\end{gather}
for any $\vp\in V_2(Q)$,
and inequality \eqref{ineq3} and similarly to it,
the following bound holds
\begin{gather*}
\|\beta\hw\|_{[V_{2;3_*}(Q)]^*}
 \leq C\Big[\big(\inf_{M_0}\|\beta_1\|_{L^{q,r}(Q)}+\|\lan\beta_2\ran_\Omega\|_{L^1(0,T)}\big)\|\hw\|_{L^\infty(Q)}
\\[1mm]
 +\inf_{M_1}\|I^{\lan 3\ran}\beta_2\|_{L^{q,r}(Q)}
\big(\|D\hw\|_{L^2(Q)}+\|\hw\|_{L^\infty(Q)}\big)\Big].
\end{gather*}
In the case $\|D\hw\|_{L^{2,\infty}(Q)}\leq K$, the norm of
$I^{\lan 3\ran}\beta_2$ can be weakened down to $\|I^{\lan 3\ran}\beta_2\|_{L^2(Q)}$.

\par Further, we bound the fourth term on the right-hand side of (\ref{hEE1}) that is a rather lengthy procedure.
To this end, using formula \eqref{dthuG} for $D_t\hu_\G$ and the similar formula for  $D_tu_\G$ (with $\beta=0$), we first accomplish the following transformations
\begin{gather*}
\hat{F}-F-\hsi_\G\b=(g[x_e]-(D_t\hu)_\G+\delta_{m3}D\hsi_\G)(\hw-w)
 +(\hg[\hx_e]-g[x_e])\hw +\hat{f}[\hx_e]-f[x_e]
\\[1mm]
-[(D_t\hu)_\G-(D_tu)_\G-\delta_{m3}D(\hsi_\G-\sigma_\G)]w
-\delta_{m1}\big[\tfrac{1}{\hV}\hu_{X0}(\hw+IP_{\het}\beta)\hw-\tfrac{1}{V}u_{X0}w^2\big]
\\[1mm
]+\delta_{m1}\big\{[2\nu D(\hw-w)-k(\hth-\theta)]\hrho D\hu_\G
\\[1mm]
+(2\nu Dw-k\theta)(\hrho D\hu_\G-\rho Du_\G)
+\nu\big[\hat{\eta}(\hrho D\hu_\G)^2-\eta(\rho Du_\G)^2\big]+\beta\nu\hrho D\hu_\G\big\}-\beta\hsi_\G 
\end{gather*}
and also, for $m=1$,
\[
\int_{Q_t}[D(\hw-w)](\hrho D\hu_\G)y\,dx d\tau =-\int_{Q_t}(\hw-w)(\hrho D\hu_\G)Dy\,dx d\tau,
\]
for any $y\in L^2(Q)$ with $Dy\in L^2(Q)$, using the independence of $\hrho D\hu_\G$ on $x$, and
\[
\hat{\eta}(\hrho D\hu_\G)^2-\eta(\rho Du_\G)^2
 =(D\hu_\G+Du_\G)(\hrho D\hu_\G-\rho Du_\G)
 +(\hat{\eta}-\eta)(\hrho D\hu_\G)\rho Du_\G.
\]
Owing to inequality (\ref{in5a}) they imply the bound
\begin{gather*}
 \|\hat{F}-F-\hsi_\G\b\|_{[H^{2,1;\hrho,3}(Q_t)]^*}
 \leq CK_0\big(\|g[x_e]\|_{L^2(Q)}+\|D_t\hu_\G\|_{L^2(Q)}
 +\delta_{m3}\|D\hsi_\G\|_{L^2(Q)}\big)\|\hw-w\|_{L^2(Q_t)}
\\[1mm]
+\|(\hg[\hx_e]-g[x_e])\hw\|_{[H^{2,1;\hrho,3}(Q_t)]^*}
 +\|\hat{f}[\hx_e]-f[x_e]\|_{[H^{2,1;\hrho,3}(Q_t)]^*}
\\[1mm]
 +CK_0\big(\|(D_t\hu)_\G-(D_tu)_\G\|_{L^1(Q_t)}+\delta_{m3}\|D(\hsi_\G-\sigma_\G)\|_{L^1(Q)}\big)
 \|w\|_{L^\infty(Q)}
\\[1mm]
+\delta_{m1}\big(\|\tfrac{1}{\hV}\hu_{X0}\hw^2-\tfrac{1}{V}u_{X0}w^2\|_{L^1(Q_t)}+\|\tfrac{1}{\hV}(\hu_X-\hu_0)\|_{L^\infty(0,T)}\|IP_{\het}\beta\|_{L^1(Q)} \|\hw\|_{L^\infty(Q)}\big)
\\[1mm] 
+CK_0\delta_{m1}\big\{
\big(2\nu\|\hw-w\|_{L^{2,1}(Q_t)}
 +k\|\hth-\theta\|_{L^1(Q_t)}\big)\|\hrho D\hu_\G\|_{L^\infty(Q)}
 \\[1mm]
+CK_0\big[2\nu
\|Dw\|_{L^2(Q)}+k\|\theta\|_{L^2(Q)}
 +\nu C\big(\|D\hu_\G\|_{L^\infty(Q)}+\|Du_\G\|_{L^\infty(Q)}\big)\big]
 \|\hrho D\hu_\G-\rho Du_\G\|_{L^2(Q_t)}
\\[1mm]
+CK_0\nu\|\hat{\eta}-\eta\|_{L^1(Q_t)}
 \|\hrho D\hu_\G\|_{L^\infty(Q)}\|\rho Du_\G\|_{L^\infty(Q)}\big\}
 +\|\beta(\delta_{m1}\nu\hrho D\hu_\G-\hsi_\G)\|_{[H^{2,1;\hrho,3}(Q)]^*}.
\end{gather*}

\par Next we bound the right-hand side of this inequality term by term, and first we use the bound
\begin{gather*}
\|g[x_e]\|_{L^2(Q)}+\|(D_t\hu)_\G\|_{L^2(Q)}
 +\delta_{m3}\|D\hsi_\G\|_{L^2(Q)}
\\[1mm]  
 \leq \|\bar{g}\|_{L^2(Q)}+ C(\|D_t\mathbf{\hu}_b\|_{L^2(0,T)}+\delta_{m3}\|\mathbf{\hp}_b\|_{L^2(0,T)})\leq C_1N,
\end{gather*}

\par We apply the formulas
\begin{gather} 
\hg[\hx_e]-g[x_e]=g[\hx_e]-g[x_e]+\hg[\hx_e]-g[\hx_e]=g[\hx_e]-g[x_e]+g_1+g_2,
\label{hgg}\\[1mm]
\int_{Q_t}g_2\hw y\,dx dt
 =-\int_{Q_t}(I^{\lan m\ran}g_2)[(D\hw)y+\hw Dy]\,dx dt
+\de_{m3}\int_{Q_t}\lan g_2\ran_\Omega\hw y\,dx dt,
\nonumber
\end{gather}
for any $y\in L^\infty(Q)$ with $Dy\in L^{2,\infty}(Q)$ (the latter one follows from formulas (\ref{fpsi})), and get
\begin{gather*}
\|(\hg[\hx_e]-g[x_e])\hw\|_{[H^{2,1;\hrho,3}(Q_t)]^*}
 \leq CK_0\big(\|g[\hx_e]-g[x_e]\|_{L^1(Q_t)}
+\|g_1\|_{L^1(Q)}\big)\|\hw\|_{L^\infty(Q)}
\\[1mm]
+CK_0\big(\|I^{\lan m\ran}g_2\|_{L^2(Q)}\|\hw\|_{V_2(Q)}
 +\de_{m3}\|\lan g_2\ran_\Omega\|_{L^1(0,T)}\|\hw\|_{L^{1,\infty}(Q)}\big).
\end{gather*}
In the case $\|D\hw\|_{L^{2,\infty}(Q)}\leq K$, the norm of $I^{\lan m\ran}g_2$ can be weakened down to
$\|I^{\lan m\ran}g_2\|_{L^{2,1}(Q)}$.

\par Formula (\ref{hgg}), with $g$ replaced with $f$, and inequality \eqref{in5a} imply
\[
 \|\hat{f}[\hx_e]-f[x_e]\|_{[H^{2,1;\hrho,3}(Q_t)]^*}
 \leq CK_0\|f[\hx_e]-f[x_e]\|_{L^1(Q_t)}
 +\|\hat{f}[\hx_e]-f[\hx_e]\|_{[H^{2,1;\hrho,3}(Q)]^*}.
\]

\par Clearly, by definition, in the case $m=1$ we have 
\[
(D_t\hu)_\G-(D_t u)_\G=D_t\hu_0-D_tu_0
+\tfrac{1}{\hV}(I\het)D_t(\hu_{X0}-u_{X0})+\big(\tfrac{1}{\hV}I\het-\tfrac{1}{V}I\eta\big)D_tu_{X0}.
\]
Applying also bound \eqref{hVIeta}, we get
\[
\|(D_t\hu)_\G-(D_t u)_\G\|_{L^1(Q_t)}
 \leq C\big(\|D_t(\mathbf{\hu}_b-\mathbf{u}_b)\|_{L^1(0,T)}+N\|\hat{\eta}-\eta\|_{L^{1,\infty}(Q_t)}\|D_t\mathbf{u}_b\|_{L^1(0,T)}\big).
\]
In the cases $m=2,3$, we simply have $\|(D_t\hu)_\G-(D_t u)_\G\|_{L^1(Q_t)}\leq\|D_t(\hu_X-u_X)\|_{L^1(0,T)}$.

\par We also can write
\begin{gather*}
\tfrac{1}{\hV}\hu_{X0}\hw^2-\tfrac{1}{V}u_{X0}w^2=\big(\tfrac{1}{\hV}\hu_{X0}-\tfrac{1}{V}u_{X0}\big)w^2+\tfrac{1}{V}\hu_{X0}(\hw+w)(\hw-w),
\\[1mm]
 \tfrac{1}{\hV}\hu_{X0}-\tfrac{1}{V}u_{X0}
 =\tfrac{1}{V}(\hu_{X0}-u_{X0})
 -\hu_{X0}\tfrac{1}{\hV V}(\hat{V}-V\big).
\end{gather*}
Consequently we obtain
\begin{gather*}
\|\tfrac{1}{\hV}\hu_{X0}\hw^2-\tfrac{1}{V}u_{X0}w^2\|_{L^1(Q_t)}
\\[1mm]
\leq  \big(NX\|\mathbf{\hu}_b-\mathbf{u}_b\|_{L^2(0,T)}
+NK_0\|\mathbf{u}_b\|_{L^2(0,T)}\|\hat{\eta}-\eta\|_{L^{1,\infty}(Q_t)}\big)\|w\|_{L^\infty(Q)}^2
\\[1mm] 
+N\|\mathbf{u}_b\|_{L^\infty(0,T)}\big(\|\hw\|_{L^\infty(Q)}+\|w\|_{L^\infty(Q)}\big)\|\hw-w\|_{L^1(Q_t)}.
\end{gather*}
Also quite similarly, owing to formulas \eqref{DhuG} for $D\hu_\G$, we have
\begin{gather*}
\|\hrho D\hu_\G-\rho Du_\G\|_{L^2(Q_t)}
 \leq N\sqrt{X}\|\mathbf{\hu}_b-\mathbf{u}_b\|_{L^2(0,T)}
+N\tfrac{K_0}{\sqrt{X}}\|\mathbf{u}_b\|_{L^2(0,T)}\|\hat{\eta}-\eta\|_{L^{1,\infty}(Q_t)}.
\end{gather*}
\par In addition, we get
\[
\|IP_{\het}\beta\|_{L^1(Q)}\leq \|I^{\lan 3\ran}\beta\|_{L^1(Q)}+C \|\lan \beta\ran_\Om\|_{L^1(0,T)}.
\]
Similarly to formula \eqref{in5b} for $m=3$, we have
\begin{gather*}
 \int_Q\beta(z-\hsi_\G)y\,dx dt
 =\int_Q\lan\beta\ran_\Om(z-\hsi_\G)y\,dx dt
  -\int_Q\,(I^{\lan 3\ran}\beta)\big[-(D\hsi_\G)y+(z-\hsi_\G)Dy\big]\,dx dt
\end{gather*}
for $z:=\delta_{m1}\nu\hrho D\hu_\G$ and any $y\in H^{2,1;\hrho,3}(Q)$.
Therefore, similarly to inequality \eqref{ineq3} we have
\begin{gather*}
 \|\beta(\delta_{m1}\nu\hrho D\hu_\G-\hsi_\G)\|_{[H^{2,1;\hrho,3}(Q)]^*}
\\[1mm]
 \leq C\big(\|\lan\beta\ran_\Om\|_{L^1(0,T)}
 +\inf_{M_0}\|I^{\lan 3\ran}\beta\|_{L^{q,r}(Q)}\big)
\big(\delta_{m1}\|\hrho D\hu_\G\|_{L^\infty(Q)}
+\|D\hsi_\G\|_{L^\infty(Q)} +\|\hsi_\G\|_{L^\infty(Q)}\big).
\end{gather*}

\par Now we turn to the fifth term on the right-hand side of (\ref{hEE1}).
Owing to the identity
\begin{gather*}
\int_{Q}z(Du)y\,dx dt
 =\int_0^T[zuy]\big|_{x=0}^{x=X}\,dt
 -\int_{Q}\big\{(Dz)y+zDy\big\}u\,dx dt,
\end{gather*}
for $z:=\hsi_\G-\sigma_\G$ and any $y\in L^\infty(Q)$ with $Dy\in L^{2,\infty}(Q)$,
together with inequality (\ref{in5a}), we get
\begin{gather*}
\|(\hsi_\G-\sigma_\G)Du
 -(D_t\hsi_\G)(\hat{\eta}-\eta)\|_{[H^{2,1;\hrho,3}(Q_t)]^*}
\\[1mm]
\leq CK_0\big[\|\mathbf{\hp}_b-\mathbf{p}_b\|_{L^1(0,T)}\big(\|u|_{x=0}\|_{L^\infty(0,T)}+\|u|_{x=X}\|_{L^\infty(0,T)}\big)
\\[1mm]
+\big(\|D(\hsi_\G-\sigma_\G)\|_{L^1(Q)}
 +\|\hsi_\G-\sigma_\G\|_{L^{2,1}(Q)}\big)\|u\|_{L^\infty(Q)}
+\|D_t\hsi_\G\|_{L^{\infty,1}(Q)}\|\hat{\eta}-\eta\|_{L^{1,\infty}(Q_t)}\big].
\end{gather*}
Recall that since $\|u\|_{L^\infty(Q)}\leq K_1$ and $Du\in L^1(Q)$, one has $\|u|_{x=0}\|_{L^\infty(0,T)}+\|u|_{x=X}\|_{L^\infty(0,T)}\leq 2K_1$.
The involved difference $\hsi_\G-\sigma_\G$ has already been bounded above in (\ref{hsgsg}). 

\par For the last sixth term on the right-hand side of (\ref{hEE1}), the following bound holds
\[
\|G\|_{L^2(Q_t)}=\|\hsi_\G(\hat{\eta}-\eta)-\hsi_\G|_{t=0}(\hat{\eta}^0-\eta^0)\|_{L^2(Q_t)}
\leq\|\hsi_\G\|_{C(\bar{Q})}2T^{1/2}\|\hat{\eta}-\eta\|_{C(0,t;L^2(\Om))}.
\]

\par Eventually bound (\ref{hEE1}) leads to the following one
\begin{gather}
\|\hat{E}-E\|_{L^2(Q_t)}
 \leq K^{(2)}(N,T)\Big(\|\hat{\eta}-\eta\|_{C(0,t;L^2(\Om))}
 +\|\hu-u\|_{L^2(Q_t)}
 +\|\hth-\theta\|_{L^{2,1}(Q_t)}
\nonumber\\[1mm]
 +\|\hat{E}^0-E^0\|_{H^{-1;3}}
 +\|\mathbf{\hu}_b-\mathbf{u}_b\|_{W^{1,1}(0,T)}
 +\|\mathbf{\hp}_b-\mathbf{p}_b\|_{L^1(0,T)}
 +\|\bm{\hpi}_b-\bm{\pi}_b\|_{L^1(0,T)}
\nonumber\\[1mm]
 +\inf_{M_0}\|\beta_1\|_{L^{q,r}(Q)}
 +\inf_{M_1}\|I^{\lan 3\ran}\beta_2\|_{L^{q,r}(Q)}
 +\|\lan\beta_2\ran_\Omega\|_{L^1(0,T)}
 +\|\g\|_{[V_{2;3_*}(Q)]^*}
\nonumber\\[1mm]
 +\|\hg[\hx_e]-g[\hx_e]\|_{L^1(Q_t)}
 +\|g_1\|_{L^1(Q)}
 +\|I^{\lan m\ran}g_2\|_{L^2(Q)}+\de_{m3}\|\lan g_2\ran_\Omega\|_{L^1(0,T)}
\nonumber\\[1mm]
 +\|f[\hx_e]-f[x_e]\|_{L^1(Q_t)}
 +\|\hat{f}[\hx_e]-f[\hx_e]\|_{[H^{2,1;\hrho,3}(Q)]^*}\Big),
\ \ 0<t\leq T.
\label{hEE2}
\end{gather}
\par The formula
$c_V(\hth-\theta)=\hat{E}-E-\half(\hw^2-w^2)$
and the above bound for $\|\hw^2-w^2\|_{L^2(Q_t)}$ allow us to replace $\hat{E}-E$ by $\hth-\theta$ on the left in \eqref{hEE2}.
We divide such bound by $2K^{(2)}(N,T)$ and add the result to the improved bound (\ref{huu2}), without the term $\|\hat{\eta}-\eta\|_{L^{2,4}(Q_t)}$ on the right, proved at the previous step of the proof.
For terms with $\beta=\beta_1+\beta_2$, we apply inequalities \eqref{ineq3} for $\beta=\beta_1$ and 
\begin{gather*}
\|I_t\lan\beta\ran_\Om\|_{L^\infty(0,T)}\leq|\lan\beta\ran_\Om\|_{L^1(0,T)}\leq\tfrac{1}{X}\|\beta\|_{L^1(Q)},
\\
\|\beta_2\|_{[V_{2;\hrho,m_*}(Q)]^*}\leq\sqrt{X}|\lan\beta_2\ran_\Om\|_{L^1(0,T)}+\|I^{\lan 3\ran}\beta_2\|_{L^2(Q)},
\end{gather*}
for $\beta=\beta_1,\beta_2$, the last of which follows from identity \eqref{id b w vp} with $\hw$ replaced with 1.
Coarsening some norms and applying also the bounds
\[
 \|g[\hx_e]-g[x_e]\|_{L^1(Q_t)}+\|f[\hx_e]-f[x_e]\|_{L^1(Q_t)}
 \leq 2\|b|\hx_e-x_e|\|_{L^1(Q_t)}
 \leq 2I_t\big(\|b\|_{L^{q_e'}(\Om)}\|\hx_e-x_e\|_{L^{q_e}(\Om)}\big),
\]
where $b:=\|D_\chi g\|_{L^{\infty}(-K',K')}+\|D_\chi f\|_{L^{\infty}(-K',K')}$, with $K'$ such that $|x_e|\leq K'$ and $|\hx_e|\leq K'$, and $\|b\|_{L^{q_e',1}(Q)}\leq K_2$,
see assumption (\ref{ggff}), we derive
\begin{gather}
\|\hat{\eta}-\eta\|_{C(0,t;L^2(\Om))}
 +\|\hu-u\|_{L^2(Q_t)\cap L^\infty(0,t;H^{-1;m})}
 +\|\hth-\theta\|_{L^2(Q_t)}
 \leq K^{(3)}(N,T)\big(\|\hth-\theta\|_{L^{2,1}(Q_t)}
\nonumber\\[1mm]
+I_t\big(\|b\|_{L^{q_e'}(\Om)}\|\hx_e-x_e\|_{L^{q_e}(\Om)}\big)
 +\|\hat{\eta}^0-\eta^0\|_{L^2(\Om)}
 +\|\hu^0-u^0\|_{H^{-1;m}}
 +\|\hat{E}^0-E^0\|_{H^{-1;3}}
\nonumber\\[1mm]
 +\|\mathbf{\hu}_b-\mathbf{u}_b\|_{W^{1,1}(0,T)}
 +\|\mathbf{\hp}_b-\mathbf{p}_b\|_{L^1(0,T)}
 +\|\bm{\hpi}_b-\bm{\pi}_b\|_{L^1(0,T)}
\nonumber\\[1mm]
 +\inf_{M_0}\|\beta_1\|_{L^{q,r}(Q)}
 +\inf_{M_1}\|I^{\lan 3\ran}\beta_2\|_{L^{q,r}(Q)}
 +\|\lan\beta_2\ran_\Omega\|_{L^1(0,T)}
 +\|\g\|_{[V_{2;3_*}(Q)]^*}
\nonumber\\
 +\|g_1\|_{L^1(Q)}
 +\|I^{\lan m\ran}g_2\|_{L^2(Q)}+\de_{m3}\|\lan g_2\ran_\Omega\|_{L^1(0,T)}
 +\|\hat{f}[\hx_e]-f[\hx_e]\|_{[H^{2,1;\hrho,3}(Q)]^*}\big),
\ 0<t\leq T.
\label{hzz3}
\end{gather}
Due to Lemma \ref{lem1} this bound implies the improved one, without the term $\|\hth-\theta\|_{L^{2,1}(Q_t)}$ on the right.

\par The formula
\[
 \hat{E}^0-E^0=\hat{e}^0-e^0-(\hu^0-u^0)\hu_\G(0)-u^0[\hu_\G(0)-u_\G(0)]
 +\half[\hu_\G(0)+u_\G(0)][\hu_\G(0)-u_\G(0)]
\]
implies Remark \ref{rem1}, Item 1, with respect to the improved bound \eqref{hzz3} and consequently (looking ahead) to the resulting bound (\ref{hzz}).

\smallskip\par 4. Now we bound $\hx_e-x_e$. Applying the operator $I_t$ to equations $D_t\hx_e=\hu$ and $D_t\het=D\hu+\beta$, we derive
\[
 \hx_e=\hx_e^0+I_t\hu,\ \ DI_t\hu=\hat{\eta}-\hat{\eta}^0-I_t\beta.
\]
Since $\hx_e^0=I\hat{\eta}^0+\beta_e$ and $I_t\hu=I_t\lan\hu\ran_\Om+I^{\lan 1\ran}D I_t\hu$ (see (\ref{idp})), the following formula holds
\[
 \hx_e=\lan I\hat{\eta}^0\ran_\Om+\beta_e+I_t\lan\hu\ran_\Om
 +I^{\lan 1\ran}\hat{\eta}-I^{\lan 1\ran}I_t\beta.
\]
This formula and the similar one for problem $\mathcal{P}_m$ (with $\beta_e=\beta=0$) imply the bound
\begin{gather*}
\|\hx_e-x_e\|_{L^{q_e,\infty}(Q_t)}
 \leq C(\|\hat{\eta}-\eta\|_{L^{1,\infty}(Q_t)}+
\|\hu-u\|_{L^1(Q_t)}+\|\hat{\eta}^0-\eta^0\|_{L^1(\Om)})
\\[1mm]
+\|\beta_e\|_{L^{q_e}(\Om)}
 +\|I_t I^{\lan 1\ran}\beta\|_{L^{q_e,\infty}(Q)}.
\end{gather*}
Clearly we also have 
\begin{gather}
\|I_t I^{\lan 1\ran}\beta\|_{L^{q_e,\infty}(Q)}\leq \|I^{\lan 1\ran}\beta\|_{L^{q_e,1}(Q)}\leq C\|\beta\|_{L^1(Q)}.
\label{I1beta}
\end{gather}

\par Due to these bounds the improved bound (\ref{hzz3}) proved in the previous item, without the term $\|\hth-\theta\|_{L^{2,1}(Q_t)}$ on the right,
remains valid after adding the terms $\|\hx_e-x_e\|_{L^{q_e,\infty}(Q_t)}$ and $C\big(\|\beta_e\|_{L^{q_e}(\Om)}+\|I_t I^{\lan 1\ran}\beta_2\|_{L^{q_e,\infty}(Q)}$\big) to its left-
and right-hand sides, respectively.
Finally, the application to the result of the Gronwall--Bellman lemma completes the proof of bound (\ref{hzz}), but without the last term on the left.

\par Note that, in addition to inequalities  \eqref{I1beta}, we have
\[
\|I^{\lan 1\ran}\beta\|_{L^{q_e,1}(Q)}\leq C\|I\beta\|_{L^{q_e,1}(Q)}\leq C\|I^{\lan 3\ran}\beta\|_{L^{q_e,1}(Q)}
+C_1\|\lan\beta\ran_\Om\|_{L^1(0,T)},
\]
and the term $\|I_t I^{\lan 1\ran}\beta_2\|_{L^{q_e,\infty}(Q)}$ in bound \eqref{hzz} can be omitted for $q_e=2$ or considering only the values $q\in [q_e,\infty]$ in the definition of the set $M_1$.

\smallskip\par 5. By virtue of formula (\ref{lne}) for $\nu\ln\hat{\eta}$ and the similar formula for problem $\mathcal{P}_m$ we have
\[
 |I_t(\hsi-\sigma)|\leq \nu K_0\big(|\hat{\eta}-\eta|+|\hat{\eta}^0-\eta^0|\big)+|I_t(\hp-p)|,
\]
Consequently (using once again formula \eqref{hpp} for $\hp-p$)
\[
 \|I_t(\hsi-\sigma)\|_{C(0,T;L^2(\Om))}
 \leq K\big(\|\hat{\eta}-\eta\|_{C(0,T;L^2(\Om))}
 +\|\hth-\theta\|_{L^{2,1}(Q)}\big)
\]
that completes the proof.
In the cases $m=2,3$, formula (\ref{its2}) could be also applied.
\end{proof}

\par Now we derive  useful corollaries on the H\"{o}lder continuity on the data for the solution in stronger norms, in general, under some additional assumptions on the data; we apply them below.
Denote shortly by $\Delta$ the sum of norms in brackets on the right in bound \eqref{hzz}.

\begin{corollary}
\label{th3cor1}
Let the hypotheses of Theorem \ref{th3} be valid and  $\|\bar{g}\|_{L^{2,1}(Q)}\leq N$ in condition $(\hat{C}_2)$. Then the bounds hold
\be
 \|\hu-u\|_{L^{\infty,2}(Q)}\leq K(N)\De^{1/2},\ \ 
 \|I_t(\hsi-\sigma)\|_{C(\bar{Q})}\leq K(N)\De^{1/2}.
\label{cor1f1}
\ee
If $\|\bar{g}\|_{L^{\infty,1}(Q)}\leq N$ in conditions $(C_2)$ and $(\hat{C}_2)$, then the partially stronger bound holds
\be
\|I_t(\hsi-\sigma)\|_{C(\bar{Q})}\leq K(N)\De^{2/3}.
\label{cor1f2}
\ee
\par If $D\hth\in L^2(Q)$ and
$\|D\hth\|_{L^2(Q)}\leq \hat{K}(N)$, then the additional bounds hold
\begin{gather}
 \|\hth-\theta\|_{L^{\infty,2}(Q)}\leq K(N)\De^{1/2},\ \
\label{cor1f3}\\[1mm]
 \|\hat{\eta}-\eta\|_{L^\infty(Q)}
 \leq K(N)\big(\|\hat{\eta}^0-\eta^0\|_{L^\infty(\Om)}+\De^{1/2}\big).
\label{cor1f4}
\end{gather}
\end{corollary}
\begin{proof}
The following multiplicative inequalities are well-known
\be
\|\alpha\|_{C(\bar{\Om})}
 \leq C\|\alpha\|_{L^2(\Om)}^{1/2}\|\alpha\|_{W^{1,2}(\Om)}^{1/2},\ \
\|\alpha\|_{C(\bar{\Om})}
 \leq C\|\alpha\|_{L^2(\Om)}^{2/3}\|\alpha\|_{W^{1,\infty}(\Om)}^{1/3}.
\label{cor1f5}
\ee
Also formula (\ref{dits}) for $\hu$ implies the bound
\be
\|DI_t\hsi\|_{L^{q,\infty}(Q)}
 \leq X^{1/q}(\|\hu\|_{L^\infty(Q)}
 +\|\hu^0\|_{L^\infty(\Om)})+\|\bar{g}\|_{L^{q,1}(Q)},\ \
 1\leq q\leq\infty,
\label{cor1f6}
\ee
and the similar bound for $DI_t\sigma$ holds as well.
\par Bounds (\ref{cor1f1}) and (\ref{cor1f3}) follow from Theorem \ref{th3} due to the first inequality (\ref{cor1f5}) and bound (\ref{cor1f6}) for $q=2$.
Bound (\ref{cor1f2}) is valid due to the second inequality (\ref{cor1f5}) and bound (\ref{cor1f6}) for $q=\infty$.

\par Equality (\ref{lne}) for $\ln\hat{\eta}$, the similar one for problem $\mathcal{P}_m$ and formula (\ref{hpp}) for $\hat{p}-p$ imply
\begin{gather*}
\nu K_0^{-1}\|\hat{\eta}-\eta\|_{L^\infty(Q_t)}
 \leq\nu K_0\|\hat{\eta}^0-\eta^0\|_{L^\infty(\Om)}
 +\|I_t(\hsi-\sigma)\|_{L^\infty(Q)}
 +kK_0\|\hth-\theta\|_{L^{\infty,1}(Q)}+
\nonumber\\[1mm]
+kK_0^2I_t(\|\theta\|_{L^\infty(\Om)}\|\hat{\eta}-\eta\|_{L^\infty(\Om)}),\ \ 0<t\leq T.
\end{gather*}
Then the Gronwall--Bellman lemma leads to the bound
\[
 \|\hat{\eta}-\eta\|_{L^\infty(Q)}\leq
 K\big(\|\hat{\eta}^0-\eta^0\|_{L^\infty(\Om)}
 +\|I_t(\hsi-\sigma)\|_{L^\infty(Q)}
 +\|\hth-\theta\|_{L^{\infty,1}(Q)}\big),
 \]
and therefore the last bound (\ref{cor1f4}) follows from the previous ones.
The condition imposed on $\bar{g}$ has been not used in the bounds for $\hu-u$ and $\hth-\th$.
\end{proof}

\par 
Define the Banach space $W(Q)$ of functions $w\in H^1(Q)$ having the finite norm
$\|w\|_{W(Q)}:=\|w\|_{L^2(Q)}+\|D_tw\|_{L^2(Q)}+\|Dw\|_{L^{2,\infty}(Q)}$.
Let $\zeta\in H^1(0,T)$ be a cut-off function such that $\|D_t\zeta\|_{L^2(0,T)}\leq N$ and $\zeta(0)=0$.
\begin{corollary}
\label{th3cor2}
Let the hypotheses of Theorem \ref{th3} be valid.
\par 1. Let in addition the conditions on the data
\[
\|\zeta^2\bar{f}\|_{L^2(Q)}\leq N,\ \
\|D_t(\zeta\bu_b)\|_{L^{4/3}(0,T)}
+\|\zeta^2\bm{\pi}_b\|_{W^{1,1}(0,T)}\leq N
\]
as well as $\|\zeta\hu\|_{W(Q)}+\|\zeta^2\hth\|_{W(Q)}\leq \hat{K}(N)$ be valid. Then the following internal in $t$ bounds hold
\begin{gather}
\|\zeta(\hu-u)\|_{C(0,T;L^2(\Om))}+\|\zeta^2(\hth-\theta)\|_{C(0,T;L^2(\Om))}
 \leq K(N)\De^{1/2},
\label{cor2f1}\\[1mm]
\|\zeta(\hu-u)\|_{C(\bar{Q})}+\|\zeta^2(\hth-\theta)\|_{C(\bar{Q})}
 \leq K(N)\De^{1/4}.
\label{cor2f2}
\end{gather}

\par 2. Let in addition the stronger conditions on the data
\[
\|Du^0\|_{L^2(\Om)}+\|D\theta^0\|_{L^2(\Om)}\leq N,\,\
\|\bar{f}\|_{L^2(Q)}\leq N,\,\
 \|D_t\bu_b\|_{L^{4/3}(0,T)}+\|\bm{\pi}_b\|_{W^{1,1}(0,T)}\leq N,
\]
the conjunction conditions $u_0(0)=u^0(0)$ ($m=1$) and $u_X(0)=u^0(X)$ ($m=1,2$)
as well as the bound $\|\hu\|_{W(Q)}+\|\hth\|_{W(Q)}\leq \hat{K}(N)$ be valid. 
Then the following bounds hold
\begin{gather}
\|\hu-u\|_{C(0,T;L^2(\Om))}+\|\hth-\theta\|_{C(0,T;L^2(\Om))}
 \leq K(N)\De^{1/2},
\label{cor2f3}\\[1mm]
\|\hu-u\|_{C(\bar{Q})}+\|\hth-\theta\|_{C(\bar{Q})}
 \leq K(N)\De^{1/4}.
\label{cor2f4}
\end{gather}
\end{corollary}
\begin{proof}
According to \cite[Theorems 5.1 and 5.2]{AZRM99} (or see \cite{ZSprin17}), under assumptions of Items 1 and 2 (taking into account that $\|\bar{g}\|_{L^2(Q)}\leq N$) the additional regularity of the solution to problem $\mathcal{P}_m$ takes place 
$\|\zeta u\|_{W(Q)}+\|\zeta^2\th\|_{W(Q)}\leq K(N)$ and $\|u\|_{W(Q)}+\|\th\|_{W(Q)}\leq K(N)$, respectively. 

\par Next, the following multiplicative inequalities hold
\begin{gather*}
\|w\|_{C(0,T;L^2(\Om))}
 \leq C\|w\|_{L^2(Q)}^{1/2}(\|w\|_{L^2(Q)}+\|D_tw\|_{L^2(Q)})^{1/2},
\\[1mm]
\|w\|_{C(\bar{Q})}
 \leq C\|w\|_{L^2(Q)}^{1/4}\|w\|_{W(Q)}^{3/4};
\end{gather*}
see the second one in \cite{AZCMMP96a}.
These results allow one to derive bounds \eqref{cor2f1}-\eqref{cor2f4} from Theorem~\ref{th3}.
\end{proof}

\par The bounds in Corollary \ref{th3cor2} demonstrate the effect of parabolic smoothing of the differences $\hu-u$ and $\hth-\theta$ in comparison with the initial differences $\hu^0-u^0$ and $\hth^0-\theta^0$ in spite of non-smoothness in $x$ of $\hat{\eta}$ and $\eta$.
Also, in Item 1, bounds (\ref{cor2f1}) and (\ref{cor2f2}) for $\zeta\equiv 1$ (i.e., bounds (\ref{cor2f3}) and (\ref{cor2f4})) cannot be valid (since, for the initial differences, the norm $\|\cdot\|_{H^{-1;m}}$ is weaker than $\|\cdot\|_{L^2(\Om)}$ and, moreover,  $\|\cdot\|_{C(\bar{\Om})}$).
But in Item 2, for more regular $u^0$ and $\theta^0$ and (implicitly) $\hu^0$ and $\hth^0$, these bounds hold.
This is connected to the validity of multiplicative inequalities
\be
 \|y\|_{L^2(\Om)}
 \leq C\|\alpha\|_{H^1(\Om)}^{1/2}\|y\|_{H^{-1;m}}^{1/2},\ \
 \|y\|_{C(\bar{\Om})}
 \leq C\|y\|_{H^1(\Om)}^{3/4}\|y\|_{H^{-1;m}}^{1/4}\ \
 \text{for}\ \ y\in H^{1;m}.
\label{cor2f5}
\ee
The first of them follows from the equalities
\be
 \int_\Om y^2\,dx
 =\int_\Om yDI^{\lan m\ran}y\,dx+\de_{m3}\lan y\ran_\Omega\int_\Om y\,dx
 =-\int_\Om (Dy)I^{\lan m\ran}y\,dx+\de_{m3}X\lan y\ran_\Om^2
\label{cor2f6}
\ee
for $y\in H^{1;m}$,
and the second one is a consequence of the first inequalities (\ref{cor1f5}) and (\ref{cor2f5}).
Inequalities (\ref{cor2f5}) also indicate a certain sharpness in order of bounds (\ref{cor2f3}) and (\ref{cor2f4}).

\section{\normalsize\rm\hspace{3pt}\bf Bounds for the error of the two-scale homogenization}
\label{sect4}

\noindent\ref{sect4}. 1. \textbf{Auxiliary results on periodic rapidly oscillating functions}. We first give some definitions and recall auxiliary results from papers \cite{AZCMMP96b,AZCMMP98}.
For $0<\de<X$, define the difference and difference quotient in $x$
\[
 \Delta_\de y(x)=y(x+\delta)-y(x),\ \ \Delta_\delta^{(1)}y(x)=\tfrac{1}{\de}\Delta_\de y(x),\ \ 0<x<X-\de.
\]
We need the spaces $L^q(\Om;L^\infty(J))$ for $1\leq q\leq\infty$, with $J:=(0,1)$. 
They are closed subspaces (more precisely, they are isometrically isomorphic to closed subspaces) in the anisotropic Lebesgue spaces $L^{q,\infty}(J\times\Om)$, thus, 
\[
\|w\|_{L^q(\Om;L^\infty(J))}=\|w\|_{L^{q,\infty}(J\times\Om)}:=\|\|w\|_{L^\infty(J)}\|_{L^q(\Omega)}\ \  \text{for}\ \ w\in L^q(\Om;L^\infty(J)).
\]
Moreover (see \cite[Lemma 1]{AZCMMP98}), $w\in L^{q,\infty}(J\times \Om)$ belongs also 
to $L^q(\Om;L^\infty(J))$ if and only if
$\lim_{\delta\to +0}\|\Delta_\de w\|_{L^{\infty,1}(J\times (0,X-\delta))}=0$.

\par We need the notion of the approximative limit ${\rm ap}\,\lim_{\xi\to\xi_0} y(\xi)$ for a function $y$ measurable on $J$ and $\xi_0\in J$, for example, see \cite[Section 5.9]{Zi89}.
Let $x^{(\ve)}:=\{x/\ve-a_\ve\}$ for $x\in\Om$, where the curly brackets $\{\cdot\}$ mean taking the fractional part of a number, and $a_\ve$ is a fixed function a parameter $\ve\in (0,1]$ (in particular, $a_\ve=0$).
Let $\lan b\ran=\int_Jb(\xi)\,d\xi$ be the mean value over $J$.
\begin{proposition}
\label{prop on w(ve)}
1. For a strongly measurable function $w$: $\Om\to L^\infty(J)$,  the composition 
\be
w^{(\ve)}(x):={\rm ap}\,\lim\nolimits_{\xi\to x^{(\ve)}} w(\xi,x)
\label{app limit}
\ee
is defined almost everywhere on $\Om$, is measurable on $\Om$ and satisfies the bound
$|w^{(\ve)}(x)|\leq\|w(\cdot,x)\|_{L^\infty(J)}$ for almost all $x\in\Om$.

\smallskip\par 2. Let  $w\in L^q(\Om;L^\infty(J))$ for some $1\leq q\leq\infty$. Then $w^{(\ve)}\in L^q(\Om)$ and the properties hold
\begin{gather}
\|w^{(\ve)}\|_{L^q(\Omega)}\leq \|w\|_{L^q(\Om;L^\infty(J))},
\label{wve in Lq}\\
w^{(\ve)}\to \lan w\ran\ \text{weakly in}\ L^q(\Om)\ \ \text{for}\ 1\leq q<\infty,\ \ \text{or weakly-star in}\ L^\infty(\Om)\ \text{for}\  q=\infty.
\nonumber
\end{gather}
\end{proposition}
See these results in \cite{A97} and \cite[Theorem 1]{AZCMMP98}.
Recall the delicate moment that considered $w$ may not satisfy the well-known Carath\'{e}odory condition.

\par The next space and respective lemma were given (in a more general form) in \cite{AZCMMP98} (see also \cite{AZCMMP96a}).
Let the linear space $L^2L^\infty C(\mathbb{R}\times J\times\Omega)$ consist of functions $F$ such that:
\par (1) $F$ is measurable on $\mathbb{R}\times J\times\Omega$ and $F(\chi,\cdot)\in L^2(\Om;L^\infty(J))$ for all $\chi\in\mathbb{R}$;
\par (2) $F(\chi,\xi,x)$ is continuous in $\chi$  for almost all $(\xi,x)\in J\times\Om$, and, moreover, for any $a>1$
\[
\sup_{\xi\in J}\,\max_{\chi,\chi'\in [-a,a],\, |\chi-\chi'|\leq\alpha}|F(\chi,\xi,x)-F(\chi',\xi,x)|\to 0\ \ \text{as}\ \ \alpha\to+0;
\]

\par (3) for any $a>1$, the following two properties hold
\[
\|F\|_{L^2L^\infty C([-a,a]\times J\times\Omega)}
:=\|\|F\|_{C[-a,a]}\|_{L^{\infty,2}(J\times\Omega)}<\infty,\ \
\lim\limits_{\de\to +0}\|\Delta_\de F\|_{L^2L^\infty C([-a,a]\times J\times(0,X-\de))}\to 0.
\]
\begin{lemma}
\label{propsuper}
Let $F\in L^2L^\infty C(\mathbb{R}\times J\times\Omega)$,
$w\in L^\infty(\Om;L^\infty(J))$ and $|w|\leq a$ for some $a>1$.
Then the composition of $F$ and $w$ has the properties
\[
F[w](\xi,x):=F(w(\xi,x),\xi,x)\in L^2(\Om;L^\infty(J)),\ \ 
\|F[w]\|_{L^2(\Om;L^\infty(J))}\leq \|F\|_{L^2L^\infty C([-a,a]\times J\times\Omega)}.
\]
\end{lemma}

\par Finally, let $WH^{0,1;1}(J\times\Om)$ and $WH^{0,1,0;1,1,r}(J\times Q)$, $1\leq r\leq\infty$, be the spaces of functions
$y\in L^{1,\infty}(J\times\Om)$ and
$w\in L^{1,\infty,r}(J\times Q)$ having the finite norms
\begin{gather*}
 \|y\|_{WH^{0,1;1}(J\times\Om)}
 =\|y\|_{L^{1,\infty}(J\times\Om)}+\sup_{0<\delta<X}\|\Delta_\delta^{(1)}y\|_{L^1(J\times (0,X-\delta))},
\\[1mm]
 \|w\|_{WH^{0,1,0;1,1,r}(J\times Q)}
 =\|w\|_{L^{1,\infty,r}(J\times Q)}+\sup_{0<\delta<X}\|\Delta_\delta^{(1)}w\|_{L^{1,1,r}(J\times Q^\de)},\
 Q^\de:=(0,X-\delta)\times (0,T);
\end{gather*}
here we once again use the anisotropic Lebesgue spaces.
The spaces are closely related to the spaces of functions with essentially bounded variation in $x$ (for example, see \cite{EG15,K76}) and contain functions with discontinuities of a very general kind, in particular, see \cite[Ch. 10]{ABM06}, thus they are much broader than their simpler subspaces of functions such that $y,Dy\in L^1(J\times\Omega)$ and $w,Dw\in L^{1,1,r}(J\times Q)$ which cannot be discontinuous in $x$. 
The notation of spaces follows traditions in theory of functions \cite{N75}.

\par The next result slightly generalizes \cite[Propositions 2 and 3]{AZCMMP96b}, with the same proof.
It is essential below and illuminates the role of using the dual spaces $H^{-1;m}$ in Theorem \ref{th1}.
Let $R_\ve y:=y^{(\ve)}-\lan y\ran$ be the averaging error.
\begin{proposition}
\label{propeps}
1. For $y\in L^1(\Om;L^\infty(J))\cap WH^{0,1;1}(J\times\Om)$, the following bound holds
\[
 \|R_\ve y\|_{H^{-1;3}}\leq C\|IR_\ve y\|_{C(\bar{\Omega})}\leq 2C\ve\|y\|_{WH^{0,1;1}(J\times\Om)}.
\]

\par 2. Let $1\leq r\leq\infty$. For $w\in WH^{0,1,0;1,1,r}(J\times Q)$ such that $w(\cdot,t)\in L^1(\Om;L^\infty(J))$ for almost all $0<t<T$, the following bound holds
\[
 \|R_\ve w\|_{L^r(0,T;\,H^{-1;3})}\leq C\|IR_\ve w\|_{L^r(0,T;\,C(\bar{\Omega}))}\leq 2C\ve\|w\|_{WH^{0,1,0;1,1,r}(J\times Q)}.
\]
\end{proposition}
The first inequalities in both Items 1 and 2 are simple and are added for convenience.

\medskip\par\noindent \ref{sect4}. 2. 
\textbf{The Navier-Stokes equations with the discontinuous rapidly oscillating data and the Bakhvalov--Eglit two-scale homogenized equations.} 
In contrast to the previous section, now we begin with the initial functions
$\eta^0,u^0\in L^\infty(\Om;L^\infty(J))$ and
$\th^0\in L^2(\Om;L^\infty(J))$ depending also on $\xi\in J$. 
We assume that $N^{-1}\leq\eta^0$ and $0<\th^0$ on $J\times\Omega$ and $\ln\th^0\in L^1(\Om;L^\infty(J))$ as well as the bounds
\begin{gather}
\|\eta^0\|_{L^\infty(\Om;L^\infty(J))}+\|u^0\|_{L^\infty(\Om;L^\infty(J))}
+\|\th^0\|_{L^2(\Om;L^\infty(J))}+\|\ln\th^0\|_{L^1(\Om;L^\infty(J))}\leq N,
\label{cond on init data with xi}\\[1mm]
N^{-1}\leq\|\eta^0(\xi,\cdot)\|_{L^1(\Om)}+I_t(u_X-u_0)\ \ \text{on}\ \ J\times (0,T)\ \ (m=1)
\nonumber
\end{gather}
are valid.
The assumptions on $\ln\th^0$ can be omitted provided that $  N^{-1}\leq\th^0$ on $J\times\Omega$.

\par We also take the term $g$ with the properties: 
\par (1) $g$ is measurable on $\mathbb{R}\times J\times Q$ and $|g(\chi,\xi,x,t)|\leq\bar{g}(x,t)$, with 
$\|\bar{g}\|_{L^2(Q)\cap L^{\infty,1}(Q)}\leq N$; \par (2)
$g(\cdot,t)\in L^2L^\infty C(\mathbb{R}\times J\times\Omega)$ for almost all $0<t<T$; 

\par (3) $g$ is the Lipschitz-type continuous in $\chi$:
\be
|g(\chi,\xi,x,t)-g(\chi',\xi,x,t)|\leq C_1(a)\bar{g}_1(x,t)|\chi-\chi'|\ \ \text{for}\ \ \chi,\chi'\in [-a,a], (\xi,x,t)\in J\times Q, 
\label{dchig2}
\ee
for any $a>1$, with $\|\bar{g}_1\|_{L^1(Q)}\leq N$.
Let $f\geq 0$ be a term with the same properties as $g$ and $\bar{f}$ in the role of $\bar{g}$.

\par Recall also conditions $(C_3)$ excluding \eqref{meta} on the boundary data which do not depend on $\xi$.

\par Now we turn to problem $\mathcal{P}_m$, but with the rapidly oscillating initial data
$(\eta^{0,\ve},u^{0,\ve},\theta^{0,\ve},x_e^{0,\ve})$ in the role of $(\eta^0,u^0,\theta^0,x_e^0)$ as well as functions
$g^{(\ve)}$ and $f^{(\ve)}$ in the role of $g$ and $f$; here $x_e^{0,\ve}:=I\eta^{0,\ve}$ and $y^{0,\ve}=(y^0)^{(\ve)}$ for $y^0=\eta^0,u^0,\th^0,x_e^0$, see definition \eqref{app limit}. 
We denote this problem by $\mathcal{P}_m^{(\ve)}$ and its solution by  $(\eta_\ve,u_\ve,\theta_\ve,x_{e,\ve})$.
According to the above assumptions on $\eta^0,u^0$ and $\th^0$ and Proposition~\ref{prop on w(ve)} we have that $\eta^{0,\ve},u^{0,\ve}\in L^\infty(\Omega)$, $\theta^{0,\ve}\in L^2(\Omega)$, $\ln\theta^{0,\ve}\in L^1(\Omega)$ and the bounds hold
\begin{gather}
 N^{-1}\leq\eta^{0,\ve}\ \text{on}\ \Om,\ \ 
\|\eta^{0,\ve}\|_{L^\infty(\Omega)}+\|u^{0,\ve}\|_{L^\infty(\Omega)}+\|\theta^{0,\ve}\|_{L^2(\Omega)}+\|\ln\theta^{0,\ve}\|_{L^1(\Omega)}\leq N,
\label{idvebounds}\\[1mm]
N^{-1}\leq\|\eta^{0,\ve}\|_{L^1(\Om)}+I_t(u_X-u_0)\ \ \text{on}\ \ (0,T)\ \ (m=1).
\end{gather}
Due to the above assumptions on $g$,  Proposition~\ref{prop on w(ve)} and Lemma \ref{propsuper} we get that $g^{(\ve)}(\chi,x,t)$ is measurable on $\mathbb{R}\times Q$, such that there $|g^{(\ve)}(\chi,x,t)|\leq\bar{g}(x,t)$ and
satisfies the Lipschitz-type condition in~$\chi$:
\[
|g^{(\ve)}(\chi,x,t)-g^{(\ve)}(\chi',x,t)|\leq C_1(a)\bar{g}_1(x,t)|\chi-\chi'|\ \ \text{for}\ \ \chi,\chi'\in [-a,a], (x,t)\in Q, 
\]
for any $a>1$.
Here $\bar{g}$ and $\bar{g}_1$ are the same as above for $g$.
Due to the Rademacher-type theorem, for example, see \cite[Theorem 4.5]{EG15}, the last condition is equivalent to the condition on the Sobolev derivative
\[
 |D_\chi g^{(\ve)}(\chi,x,t)|\leq C_1(a)\bar{g}_1(x,t)\ \ \text{for}\ \ \chi\in [-a,a], (x,t)\in Q, 
\]
for any $a>1$.
The same properties are valid for~$f^{(\ve)}$.

\par We also introduce \textit{the Bakhvalov-Eglit two-scale homogenized system of equations}
\begin{gather}
D_t\eta=\tfrac{1}{\nu}\hsi\eta+\tfrac{k}{\nu}\hth\ \ \text{in}\ \ J\times Q,
\label{eq1g}\\[1mm]
D_t\hu=D\hsi+\lan g\ran[\hx_e],
\label{eq2g}\\[1mm]
c_VD_t\hth=D\hpi+\hsi D\hu+\lan f\ran[\hx_e],
\label{eq3g}\\[1mm]
D_t\hx_e=\hu,\ \
\hsi=\nu\hrho D\hu-\hp,\ \ \hrho=\lan\eta\ran^{-1},
\ \ \hp=k\hrho\hth,\ \ \hpi=\la\hrho D\hth,
\label{eq5g}
\end{gather}
with the solution $\eta(\xi,x,t)>0$, $\hu(x,t)$, $\hth(x,t)>0$ and $\hx_e(x,t)$; equations \eqref{eq2g}--\eqref{eq5g} are posed in $Q$.
We supplement the equations with the initial conditions
\begin{gather}
\eta|_{t=0}=\eta^0\ \ \text{in}\ \ J\times\Om,\ \ (\hu,\hth,\hx_e)|_{t=0}=(\hu^0,\hth^0,\hx_e^0)
\label{ichg}
\end{gather}
and the previous boundary conditions \eqref{bch}--\eqref{bc4h}
with the same boundary data $\mathbf{\hu}_b=\mathbf{u}_b$, $\mathbf{\hp}_b=\mathbf{p}_b$ and $\bm{\hpi}_b=\bm{\pi}_b$ as in problem $\mathcal{P}_m^{(\ve)}$.
Here the initial data are homogenized according to the formulas
\[
 \hu^0=\lan u^0\ran,\ \ \hat{e}^0:=\lan e^0\ran=\half\lan(\hu^0)^2\ran+c_V\hth^0, \ \
\ \ \hx_e^0:=I\lan\eta^0\ran,
\]
where $e^0:=\half(u^0)^2+c_V\th^0$. 
Thus, the averaging of $e^0$ rather than $\th^0$ is applied, and then explicitly we have
\be
 \hth^0=\tfrac{1}{c_V}\big(\hat{e}^0-\half\lan(\hu^0)^2\ran\big)=\tfrac{1}{2c_V}\lan(u^0-\lan u^0\ran)^2\ran+\lan\th^0\ran\geq \lan\th^0\ran>0.
\label{hth 0}
\ee
If $N^{-1}\leq\th^0$, then $N^{-1}\leq\hth^0$ as well. 
Notice that this problem is called two-scale (or quasi-averaged) since the variable $\xi$ is not completely eliminated from it.

\par But the last circumstance is not essential in the homogeneous gas case (with constant $\nu$, $k$, $c_V$ and $\lambda$) studied in this paper. 
Indeed, clearly formulas \eqref{eq5g} imply
\be
\lan\tfrac{1}{\nu}\hsi\eta+\tfrac{k}{\nu}\hth\ran=\tfrac{1}{\nu}\hsi\lan\eta\ran+\tfrac{k}{\nu}\hth
=D\hu.
\label{eq6g}
\ee
Therefore, averaging of equation \eqref{eq1g} in $\xi$ gives
\be
D_t\lan\eta\ran=D\hu.
\label{eq1g1}
\ee
Thus the functions $(\lan\eta\ran,\hu,\hth,\hx_e)$ satisfy the IBVP $\lan\mathcal{P}\ran_m$ for the closed system of equations \eqref{eq1g1} and \eqref{eq2g}--\eqref{eq5g}, with the initial conditions
\[
(\lan\eta\ran,\hu,\hth,\hx_e)|_{t=0}=(\lan\eta^0\ran,\hu^0,\hth^0,\hx_e^0)
\]
and the boundary conditions \eqref{bch}--\eqref{bc4h}.
This IBVP is simply of type $\mathcal{P}_m$ that is important below.

\par Moreover, if the solution to the IBVP $\lan\mathcal{P}\ran_m$ is found, then clearly the function $\eta=\eta(\xi,x,t)$ can be expressed by the explicit formula
\be
\eta=\hat{B}\big[\eta^0+\tfrac{k}{\nu}I_t(\hat{B}^{-1}\hth)\big]\ \ \text{in}\ \ J\times Q,\ \  \text{with}\ \ \hat{B}:=e^{\tfrac{1}{\nu}I_t\hsi},
\label{expforeta}
\ee
which appears by treating equation \eqref{eq1g} as an ODE in $t$.
This formula implies the properties
\be
(K_0(N))^{-1}\leq\eta\leq K_0(N)\ \ \text{in}\ \ J\times Q,\ \ 
\eta\in C(0,T;L^\infty(\Om;L^\infty(J)))
\label{eta in CLL}
\ee
provided that $\|I_t\hsi\|_{ L^{\infty,1}(Q)}+\|\hth\|_{ L^{\infty,1}(Q)}\leq K(N)$, see \cite{AZCMMP98} concerning the latter property.
The bound for $I_t\hsi$ and $\hth$ will be guaranteed below.

\par According to the above assumptions on $\eta^0,u^0$ and $\th^0$ and formula \eqref{hth 0} for $\hth^0$ we have 
\begin{gather*}
N^{-1}\leq\lan\eta^0\ran\,\ \text{in}\,\ Q,\ \ 
\|\lan\eta^0\ran\|_{L^\infty(\Om)}+\|\lan u^0\ran\|_{L^\infty(\Om)}+\|\lan \th^0\ran\|_{L^2(\Om)}\leq N,\ \ 
\|\hth^0\|_{L^2(\Om)}\leq \tfrac{N^2}{2c_V}\sqrt{X}+N,
\\[1mm]
N^{-1}\leq\|\lan\eta^0\ran\|_{L^1(\Om)}+I_t(u_X-u_0)\ \ \text{on}\ \ (0,T)\ \ (m=1)  
\end{gather*}
since clearly $\hth^0\leq\tfrac{1}{2c_V}\|u^0\|_{L^\infty(J)}^2+\lan\th^0\ran$ on $\Om$. 
Using also the Jensen inequality, we get 
\[
-\|\ln\th^0\|_{L^\infty(J)}\leq\ln\hth^0\leq \tfrac{1}{2c_V}\|u^0\|_{L^\infty(J)}^2+\lan\th^0\ran\ \ \text{on}\ \ \Om 
\]
and therefore
\[
\|\ln\hth^0\|_{L^1(\Om)}\leq K(N).
\]
Also $\lan g\ran$ and $\lan f\ran$, in the role of $g$ and $f$, satisfy condition $(C_2)$, with $q_e=\infty$ and $\|\bar{f}\|_{L^2(Q)}\leq N$ (instead of $\|\bar{f}\|_{L^{2,1}(Q)}\leq N$).

\par For problem $\lan\mathcal{P}\ran_m$, we impose the additional  conditions on its initial and boundary data 
\begin{gather}
 \|\lan u^0\ran\|_{H^1(\Om)}+\|\hth^0\|_{H^1(\Om)}\leq N,
\label{idhN}
\\[1mm]
\|\mathbf{u}_b\|_{H^1(0,T)}
+\|\mathbf{p}_b\|_{W^{1,1}(0,T)}+\|{\bm{\pi}}_b\|_{W^{1,1}(0,T)}\leq N,
\label{ftbdhN}
\\[1mm]
 u_0(0)=\lan u^0\ran(0)\ (m=1),\ \ u_X(0)=\lan u^0\ran(X)\ (m=1,2),
\label{conhcond}
\end{gather}
the last of which are the conjunction conditions between the boundary and initial data for $\hu$.

\par Under the listed properties and additional conditions on the data, the  unique weak solution to problem $\lan\mathcal{P}\ran_m$ is  \textit{almost strong}, i.e.,  $\hu,\hth\in W(Q)$ and
$\hsi,\hpi\in V_2(Q)$ as well as equations \eqref{eq2g}--\eqref{eq5g} are valid in $L^2(Q)$. Moreover, the bounds hold
\begin{gather}
\|\hu\|_{C(\bar{Q})}+\|\hth\|_{C(\bar{Q})}\leq C(\|\hu\|_{W(Q)}+\|\hth\|_{W(Q)})\leq CK(N),
\label{almreghsol}\\
\|\hsi\|_{V_2(Q)}+\|\hpi\|_{V_2(Q)}\leq K(N).
\label{almreghsol 2}
\end{gather}
These regularity properties of the weak solution follow from \cite[Theorem 4.2]{AZRM99}; some additional properties can be also found there.  
Notice carefully that the solution is \textit{almost} strong since neither conditions on $D\lan\eta^0\ran$ are imposed, nor the existence of $D\lan\eta\ran$ is asserted.

\par On the other hand, from formula \eqref{expforeta} for $\eta$
and its properties \eqref{eta in CLL} we get that $\eta^{(\ve)}\in C(0,T;L^\infty(\Om))$, $(K_0(N))^{-1}\leq\eta^{(\ve)}\leq K_0(N)$ in $Q$ and $\eta^{(\ve)}$ can be represented by the respective formula
\[
\eta^{(\ve)}=\hat{B}\big[\eta^{0,\ve}+\tfrac{k}{\nu}I_t(\hat{B}^{-1}\hth)\big]\ \ \text{in}\ \ Q.
\]
Consequently, $D_t\eta^{(\ve)}\in L^2(Q)$, and $\eta^{(\ve)}$ satisfies the equation and initial condition
\be
D_t\eta^{(\ve)}=\tfrac{1}{\nu}\hsi\eta^{(\ve)}+\tfrac{k}{\nu}\hth\ \ \text{in}\ \ Q,\ \ \eta^{(\ve)}|_{t=0}=\eta^{0,\ve}
\label{eq1g2}.
\ee

\par We define the functions
\[
\beta^{(\ve)}:=\tfrac{1}{\nu}\hsi\eta^{(\ve)}-\tfrac{1}{\nu}(D\hu-k\hth)
=\tfrac{1}{\nu}\hsi(\eta^{(\ve)}-
\lan\eta\ran)=\tfrac{1}{\nu}\hsi R\eta^{(\ve)},\ \
\gamma^{(\ve)}:=\tfrac{1}{\la}\hpi(\eta^{(\ve)}-\lan\eta\ran)=\tfrac{1}{\la}\hpi R\eta^{(\ve)}.
\]
Clearly $\beta^{(\ve)},\gamma^{(\ve)}\in L^{2,\infty}(Q)$, see bound \eqref{almreghsol 2}.
Then, for the functions $(\eta^{(\ve)},\hu,\hth,\hx_e)$,
from equations \eqref{eq1g2} and \eqref{eq2g}--\eqref{eq5g}
we derive the following IBVP $\widehat{\mathcal{P}}_m^{(\ve)}$ of type $\widehat{\mathcal{P}}_m$ formed by the system of equations
\begin{gather*}
D_t\eta^{(\ve)}=D\hu+\beta^{(\ve)},
\\[1mm]
D_t\hu=D(\s^{(\ve)}+\nu\rho^{(\ve)}\beta^{(\ve)})+\lan g\ran[\hx_e],
\\[1mm]
c_VD_t\hth=D(\pi^{(\ve)}+\la\rho^{(\ve)}\gamma^{(\ve)})
 +(\s^{(\ve)}+\nu\rho^{(\ve)}\beta^{(\ve)})D\hu+\lan f\ran[\hx_e],
\\[1mm]
 D_t\hx_e=\hu,\ \
\s^{(\ve)}=\nu\rho^{(\ve)} D\hu-p^{(\ve)},\ \
 \rho^{(\ve)}=\tfrac{1}{\eta^{(\ve)}},\ \
p^{(\ve)}=k\rho^{(\ve)}\hth,\ \ \pi^{(\ve)}=\la\rho^{(\ve)} D\hth
\end{gather*}
in $Q$ supplemented with the initial conditions
\[
(\eta^{(\ve)},\hu,\hth,\hx_e)|_{t=0}=(\eta^{0,\ve},\hu^0,\hth^0,\hx_e^0)
\]
and the boundary conditions \eqref{bch} with $\mathbf{\hu}_b=\mathbf{u}_b$, $\mathbf{\hp}_b=\mathbf{p}_b$ and $\bm{\hpi}_b=\bm{\pi}_b$.
Here we have used the formulas
$\hsi=\s^{(\ve)}+\nu\rho^{(\ve)}\beta^{(\ve)}$ and $\hpi=\pi^{(\ve)}+\la\rho^{(\ve)}\gamma^{(\ve)}$
or, equivalently,
\[
 \tfrac{1}{\nu}\eta^{(\ve)}(\hsi-\s^{(\ve)})=\beta^{(\ve)},\ \
 \tfrac{1}{\la}\eta^{(\ve)}(\hpi-\pi^{(\ve)})=\gamma^{(\ve)}
\]
which follow from the equalities
$\eta^{(\ve)}\s^{(\ve)}=\nu D\hu-k\hth=\lan\eta\ran\hsi$
and $\eta^{(\ve)}\pi^{(\ve)}=\lambda D\hth=\lan\eta\ran\hpi$.

\par In addition, $\hx_e^0=I\eta^{0,\ve}+\beta_{e\ve}$ with $\beta_{e\ve}:=I(\lan\eta^0\ran-\eta^{0,\ve})=-IR_\ve\eta^0$.

\par Notice that the data of this problem satisfy conditions $(\hat{C}_1)$--$(\hat{C}_3)$, and its solution satisfies  assumptions imposed on the solution to problem $\widehat{\mathcal{P}}_m$ in the previous section.

\par The next theorem on a bound for the error of the two-scale homogenization is the second main result of the paper.
\begin{theorem}
\label{th4 eps}
Let the above conditions and the following additional weak regularity of the initial functions $\eta^0,u^0,e^0$ and the terms $g$ and $f$ in~$x$ be valid  
\begin{gather}
 \|(\eta^0,u^0,e^0)\|_{WH^{0,1;1}(J\times\Om)}\leq N,
\label{idregx}
\\[1mm]
 \sup_{0<\de<X}\|(\De_\de^{(1)} g,\De_\de^{(1)} f)\|_{L^1(J\times Q^\de;\,C[-a,a])}\leq C_1(a)\ \
 \text{for any}\ \ a>1.
\label{ftregx}
\end{gather}
Then the following bound for the error of the two-scale homogenization hold
\begin{gather}
\|\eta^{(\ve)}-\eta_\ve\|_{C(0,T;L^2(\Om))}
 +\|\hu-u_\ve\|_{L^2(Q)\cap L^\infty(0,T;H^{-1;m})}
 +\|\hth-\theta_\ve\|_{L^2(Q)}
\nonumber\\[1mm]
 +\|\hx_e-x_{e,\ve}\|_{L^\infty(Q)}
 +\|I_t(\hsi-\sigma_\ve)\|_{C(0,T;L^2(\Om))}\leq K(N)\ve.
\label{errest2}
\end{gather}
\end{theorem}
\begin{proof}
1. We first establish the following weak regularity of $\eta$ in $x$:
\be
\sup_{0<\delta<X}\|\Delta_\delta^{(1)}\eta\|_{L^{1,1,\infty}(J\times Q^\de)}\leq K.
\label{etareg}
\ee

\par For $0<\de<X$, we apply the operator $\Delta_\delta$ to the explicit formula  \eqref{expforeta} for $\eta$ and get
\begin{gather*}
\Delta_\de\eta=(\Delta_\de \hat{B})\big[\eta^0+\tfrac{k}{\nu}I_t(\hat{B}^{-1}\hth)\big]+\hat{B}_{(\de)}\big\{\Delta_\de\eta^0+\tfrac{k}{\nu}I_t\big[(\Delta_\de (\hat{B}^{-1}))\hth+\hat{B}_{(\de)}^{-1}\Delta_\de\hth\,\big]\big\}
\end{gather*}
in $J\times Q^\de$, with $\hat{B}_{(\de)}(x,t):=\hat{B}(x+\de,t)$.
Since $\|I_t\hsi\|_{L^\infty(Q)}\leq K_1$ due to bound \eqref{almreghsol 2},  we have $K_2^{-1}\leq\hat{B}\leq K_2$ and consequently
\begin{gather*}
 |\Delta_\de\eta|\leq K|\Delta_\de I_t\hsi|\big(|\eta^0|+\|\hth\|_{L^{\infty,1}(Q)}\big)
+K\big(|\Delta_\de\eta^0|+\|\Delta_\de I_t\hsi\|_{L^2(0,T)}\|\hth\|_{L^2(0,T)}+\|\Delta_\de\hth\|_{L^1(0,T)}\big)
\end{gather*}
in $J\times Q^\de$.
Therefore, we obtain
\[
\|\Delta_\de^{(1)}\eta\|_{L^{1,1,\infty}(J\times Q^\de)}
 \leq K\big(\|DI_t\hsi\|_{L^{2,\infty}(Q)}\|\hth\|_{L^{\infty,1}(Q)\cap L^2(Q)}+\|\Delta_\de^{(1)}\eta^0\|_{L^1(J\times Q^\de)}+\|D\hth\|_{L^1(Q)}\big)\leq K_3
\]
owing to condition \eqref{idregx} on $\eta^0$ and bound \eqref{almreghsol}, i.e., bound \eqref{etareg} is proved.

\par Next, we prove the weak regularity of $g[\hat{x}_e]$ in $x$:
\be
\|g[\hat{x}_e]\|_{WH^{0,1,0;1}(J\times Q)}:=
\|g[\hat{x}_e]\|_{WH^{0,1,0;1,1,1}(J\times\Om\times (0,T))}\leq K.
\label{reg g in x}
\ee
Applying the obvious formula
\[
(\Delta_\de g[\hx_e])(\xi,x,t)=g(\hx_e(x+\de,t),\xi,x+\de,t)-g(\hx_e(x,t),\xi,x+\de,t)+((\Delta_\de g)[\hx_e])(\xi,x,t)
\]
in $Q^\de$ and the Lipschitz-type condition \eqref{dchig2} on $g$, we find
\[
\|\Delta_\de^{(1)}g[\hx_e]\|_{L^1(J\times Q^\de)}\leq K\|\Delta_\de^{(1)}\hx_e\|_{L^\infty(Q^\de)}
+\|\Delta_\de^{(1)}g\|_{L^1(J\times Q^\de;\,C[-a,a])},\ \ 0<\de<X,
\]
with $a=K'$ such that $|\hx_e|\leq K'$. 
Since $\hx_e=\hx_e^0+I_t\hu$ and thus $D\hx_e=\lan\eta^0\ran+I_tD\hu$, we have
\[
\sup_{0<\de<X}\|\Delta_\de^{(1)}\hx_e\|_{L^\infty(Q^\de)}\leq\|D\hx_e\|_{L^\infty(Q)}
\leq\|\lan\eta^0\ran\|_{L^\infty(\Om)}+\|D\hu\|_{L^{\infty,1}(Q)}\leq K
\]
owing to the formula $D\hu=\frac{\lan\eta\ran}{\nu}(\hsi+\hp)$, see \eqref{eq5g},  
and bound \eqref{almreghsol 2} for $\hsi$.
According to the regularity condition \eqref{ftregx} on $g$ and the bound  $|g[\hx_e]|\leq\bar{g}$ with $\|\bar{g}\|_{L^{\infty,1}(Q)}\leq N$, we get bound \eqref{reg g in x}.
Clearly the same bound holds for $f[\hx_e]$.

\smallskip\par 2. We apply Theorem \ref{th3} to the original problem $\mathcal{P}_m^{(\ve)}$ with the solution $(\eta_\ve,u_\ve,\theta_\ve,x_{e,\ve})$ and  problem $\widehat{\mathcal{P}}_m^{(\ve)}$ with the solution $(\eta^{(\ve)},\hat{u},\hat{\theta},\hat{x}_e)$ (constructed as explained above from the solution to problem $\lan\mathcal{P}\ran_m$).
In Theorem \ref{th3}, we choose the case $\beta_1=g_1=0$, $q_e=\infty$ and $(q,r)=(\infty,2)\in M_1$,  take into account Remark \ref{rem1}, Items 1 and 3 (using bound \eqref{almreghsol} for $\hu$), also apply inequalities \eqref{ineq3a} for $w=\g^{(\ve)}$  as well as \eqref{est for F} and \eqref{ineq3} for $F=R_\ve f[\hx_e]$ and derive
\begin{gather}
\|\eta^{(\ve)}-\eta_\ve\|_{C(0,T;L^2(\Om))}
 +\|\hu-u_\ve\|_{L^2(Q)\cap L^\infty(0,T;H^{-1;m})}
 +\|\hth-\theta_\ve\|_{L^2(Q)}
 +\|\hx_e-x_{e,\ve}\|_{L^\infty(Q)}
\nonumber\\[1mm]
 +\|I_t(\hsi-\sigma_\ve)\|_{C(0,T;L^2(\Om))}
\leq K\big(\|R_\ve u^0\|_{H^{-1;3}}
 +\|R_\ve e^0\|_{H^{-1;3}}
+\|\beta_{e\ve}\|_{L^\infty(\Om)}
\nonumber\\[1mm]+\|I\beta^{(\ve)}\|_{L^{\infty,2}(Q)} +\|I\gamma^{(\ve)}\|_{L^2(Q)}+\|IR_\ve g[\hx_e]\|_{L^{\infty,1}(Q)}+\|IR_\ve f[\hx_e]\|_{L^{\infty,1}(Q)}\big),
\label{th4for2}
\end{gather}
where, for example, $R_\ve g[\hx_e]=g^{(\ve)}[\hx_e]-\lan g\ran[\hx_e]=g^{(\ve)}[\hx_e]-\lan g[\hx_e]\ran$.

\par We further bound the terms on the right in inequality \eqref{th4for2}.
Since $\beta_{e\ve}=-IR_\ve\eta^0$, applying Proposition  \ref{propeps}, Item 1, and conditions \eqref{cond on init data with xi} and \eqref{idregx} on $(\eta^0,u^0,e^0)$, we obtain
\begin{gather}
 \|R_\ve u^0\|_{H^{-1;3}}
 +\|R_\ve e^0\|_{H^{-1;3}}
 +\|\beta_{e\ve}\|_{L^\infty(\Om)}
 \leq C\ve\|(u^0,e^0,\eta^0)\|_{WH^{0,1;1}(J\times\Om)}\leq CN\ve.
\label{th4for3}
\end{gather}

\par Notice that the inequality holds
\[
\|I(yz)\|_{C(\bar{\Om})}\leq C\|y\|_{W^{1,1}(\Om)}\|Iz\|_{C(\bar{\Om})}\ \ \text{for any}\ \ y\in W^{1,1}(\Om), z\in L^1(\Om)
\]
(like in \cite{AZCMMP96b}) which is proved by integrating by parts, cp. \eqref{cor2f6}.
Using sequentially this inequality,
bound \eqref{almreghsol 2} for $\hsi$ and $\hpi$ and Proposition \ref{propeps}, Item 2, for $r=\infty$, as well as the regularity bound \eqref{etareg} for $\eta$ (and its property \eqref{eta in CLL}), we derive
\begin{gather}
\|I\beta^{(\ve)}\|_{L^{\infty,2}(Q)}+\|I\gamma^{(\ve)}\|_{L^2(Q)}
 =\tfrac{1}{\nu}\|I(\hsi R_\ve\eta)\|_{L^{\infty,2}(Q)}
 +\tfrac{1}{\la}\|I(\hpi R_\ve\eta)\|_{L^2(Q)}
\nonumber\\[1mm]
 \leq C(\|\hsi\|_{L^2(Q)}+\|D\hsi\|_{L^{1,2}(Q)}+\|\hpi\|_{L^2(Q)}+\|D\hpi\|_{L^{1,2}(Q)})\|IR_\ve\eta\|_{L^\infty(Q)}
\nonumber\\[1mm]
 \leq K\ve\|\eta\|_{WH^{0,1,0;1,1,\infty}(J\times Q)}\leq K_1\ve.
\label{th4for4}
\end{gather}
Note that namely here we base on the fact that the solution to problem $\lan\mathcal{P}\ran_m$ is almost strong.

\par Due to Proposition \ref{propeps}, Item 2, for $r=1$, and the regularity bound \eqref{reg g in x} for $g[x_e]$, we have
\be
 \|IR_\ve g[\hat{x}_e]\|_{L^{\infty,1}(Q)}
 \leq 2\ve\|g[\hat{x}_e\|_{WH^{0,1,0;1}(J\times Q)}\leq K\ve.
\label{th4for5}
\ee

\par Applying bounds \eqref{th4for3}, \eqref{th4for4} and \eqref{th4for5} together with the last one with $f$ in the role of $g$ to estimate the right-hand side of bound \eqref{th4for2}, we complete the proof.
\end{proof}
\begin{corollary}
\label{th3cor3}
Let the hypotheses of Theorem \ref{th4 eps} be valid.
\par 1. The following error bounds in stronger norms hold 
\begin{gather}
 \|\eta^{(\ve)}-\eta_\ve\|_{L^\infty(Q)}+\|\hu-u_\ve\|_{L^{\infty,2}(Q)}
 +\|\hth-\theta_\ve\|_{L^{\infty,2}(Q)}
 \leq K(N)\ve^{1/2},
\label{cor3f1}\\[1mm]
 \|I_t(\hsi-\sigma_\ve)\|_{C(\bar{Q})}\leq K(N)\ve^{2/3}.
\nonumber
\end{gather}

\par 2. Also the following internal error bounds hold
\begin{gather}
\|\zeta(\hu-u_\ve)\|_{C(0,T;L^2(\Om))}
 +\|\zeta^2(\hth-\theta_\ve)\|_{C(0,T;L^2(\Om))}
 \leq K(N)\ve^{1/2},
\nonumber\\[1mm]
\|\zeta(\hu-u_\ve)\|_{C(\bar{Q})}+\|\zeta^2(\hth-\theta_\ve)\|_{C(\bar{Q})}
 \leq K(N)\ve^{1/4};
\label{cor3f4}
\end{gather}
recall that here $\zeta\in H^1(0,T)$, $\|D_t\zeta\|_{L^2(0,T)}\leq N$ and $\zeta(0)=0$.

\par 3. If $u^0=u^0(x)$ and $\th^0=\th^0(x)$ are independent of $\xi$, then the following error bounds hold
\begin{gather*}
\|\hu-u_\ve\|_{C(0,T;L^2(\Om))}+\|\hth-\theta_\ve\|_{C(0,T;L^2(\Om))}
 \leq K(N)\ve^{1/2},
\label{cor3f5}\\[1mm]
\|\hu-u_\ve\|_{C(\bar{Q})}+\|\hth-\theta_\ve\|_{C(\bar{Q})}
 \leq K(N)\ve^{1/4}.
\label{cor3f6}
\end{gather*}
\end{corollary}
\begin{proof}
These error bounds follow from the main error bound \eqref{errest2} by virtue of the corresponding bounds in Corollaries \ref{th3cor1} and \ref{th3cor2}.
In Items 1 and 2, we take into account the equality 
$(\eta^{(\ve)}-\eta_\ve)|_{t=0}=0$,  conditions $\|(\bar{g},\bar{f})\|_{L^2(Q)\cap L^{\infty,1}(Q)}\leq N$ and \eqref{ftbdhN} on $\mathbf{u}_b$ and ${\bm{\pi}}_b$ as well as bound \eqref{almreghsol} for $\hu$ and $\hth$.
In Item 3, $u^0$ and $\th^0$ coincide with $\lan u^0\ran$ and $\hth^0$, and thus they have the same properties \eqref{idhN} and \eqref{conhcond}.
\end{proof}

We emphasize that the bound for $\eta^{(\ve)}-\eta_\ve$ in the $L^\infty(Q)$--norm in \eqref{cor3f1} and the internal error bound \eqref{cor3f4} in the $C(\bar{Q})$--norm
are valid for \textit{discontinuous} initial data.
In Item 3, $\eta^0$ is still $\xi$--dependent and can be discontinuous.

\bigskip\par
\noindent 
\textbf{Acknowledgements.}
This study was supported by 
the Basic Research Program at the HSE University (Laboratory of Mathematical Methods in Natural Science). 
\renewcommand{\refname}
{\begin{center}{\normalsize \textbf{REFERENCES}} \end{center}}
\small{

}
\end{document}